\documentclass{article}
\usepackage{hyperref}
\usepackage[utf8]{inputenc}
\usepackage[english]{babel}
\usepackage[title]{appendix}
\usepackage{graphicx}
\usepackage{blkarray}
\usepackage{amsmath}
\usepackage{amssymb}
\usepackage{amsthm}
\usepackage{enumitem}
\usepackage{comment}
\usepackage[normalem]{ulem}
\usepackage{lipsum}                     
\usepackage{xargs}                      
\usepackage[pdftex,dvipsnames]{xcolor}  

\numberwithin{equation}{section}

\def\authorfont{\footnotesize}

\def\keywords#1{\par
	\vspace*{8pt}
	{\authorfont{\leftskip24pt\rightskip\leftskip
	\noindent{\it Keywords}\/:\ #1\par}}\par}

\def\ccode#1{\par		
	\vspace*{8pt}
	{\authorfont{\leftskip24pt\rightskip\leftskip
	\noindent #1\par}}\par}

\hypersetup{
  colorlinks=true,
  linkcolor=blue,
  citecolor=blue,
  urlcolor=blue,
  pdftitle={Pentagon equations, Delaunay triangulations and pure braid group invariant}
}

\title{Pentagon equations, Delaunay triangulations and pure braid group invariant}
\author{Illia E. Rohozhkin}

\theoremstyle{definition}
\newtheorem{theorem}{Theorem}[section]

\newtheorem{lemma}{Lemma}[section]

\newtheorem{definition}{Definition}[section]

\newtheorem{remark}{Remark}[section]

\date{\today}

\begin{document}

\maketitle

\begin{abstract}
We construct $(2n+1)\times (2n+1)$ matrices corresponding to a motion of points on the plane from the point of view of Delaunay triangulations. We define a homomorphism from the pure braid group on ($n+3$) strands to the general linear group $\text{GL}_{2n+1}(\mathbb{Q})$.
\end{abstract}



\section{Introduction}
In \cite{Fedoseev-Manturov-Nikonov, Manturov-Kim, Invariants-and-Pictures} a family of groups $\Gamma_n^4$ was defined. These groups depend on the parameter $n>4$ and naturally describe the motion of points on the plane. A map $f:\text{PB}_n \rightarrow \Gamma_n^4$ is constructed by using Vorono\"i diagrams and Delaunay triangulations defined from these points. Isotopic braids correspond to equivalent word in $\Gamma_n^4$ due to relations of this group. The most interesting relation of the group $\Gamma_n^4$ is the {\em pentagon relation}.

In \cite{Korepanov-pentagon}, vast families of orthogonal operators obeying pentagon relation in a direct sum of three $n$-dimensional vector spaces are constructed.

In \cite{Manturov-Zheyan}, a $2\times 2$ matrix is assigned to each $\Gamma_n^4$ generator.

In this paper we will consider a motion of $n$ points on the plane. $(2n+1)\times (2n+1)$ matrices will be defined. These matrices will satisfy specified relations. We will build a map from the pure braid group to the $\text{GL}_{2n+1}(\mathbb{Q})$ group and prove that this map is a well defined homomorphism.

\section{Basic definitions}
\subsection{Braid groups}
\label{section:braid}
Consider the lines ${y = 0, z = 1}$ and ${y = 0, z = 0}$ in $\mathbb{R}^3$ and choose $n$ points on each of these lines having abscissas $1,\dots, n$.

\begin{definition}
An $n$-strand braid is a set of $n$ non-intersecting smooth paths connecting the chosen points on the first line with the points on the second line (in arbitrary order), such that the projection of each of these paths to $Oz$ represents a diffeomorphism.
\end{definition}

These smooth paths are called {\em strands} of the braid.

An example of a braid is shown in  Fig. \ref{fig:a_braid}.

It is natural to consider braids up to isotopy in $\mathbb{R}^3$.

\begin{figure}[h]
    \centering
    \includegraphics[width=0.15\textwidth]{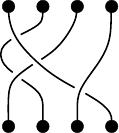}
    \caption{A braid.}
    \label{fig:a_braid}
\end{figure}

Two braids $B_0$ and $B_1$ are equal if they are {\em isotopic}; i.e., if there exists a continuous family of braids $B_t$, $\{t \in \{0,1\}\}$ of braids starting at $B_0$ and finishing at $B_1$.

The set of all $n$-strand braids generates a group. Usually this group is denoted by $\text{B}_n$.

\begin{definition}
The $n$-strand braid group $\text{B}_n$ is the group given by the presentation with $(n-1)$ generators $\sigma_1,\dots,\sigma_{n-1}$ and the following relations:
\begin{enumerate}
\item $\sigma_i\sigma_j = \sigma_j\sigma_i$ for $|i-j|\ge 2$;
\item $\sigma_i\sigma_{i+1}\sigma_i = \sigma_{i+1}\sigma_i\sigma_{i+1}$ for $1 \le i \le n-2$.
\end{enumerate}
\end{definition}

These relations are called Artin’s relations. The generators of the braid group are shown on Fig. \ref{fig:braid_group_generators}.

\begin{figure}[h]
    \centering
    \includegraphics[width=0.7\textwidth]{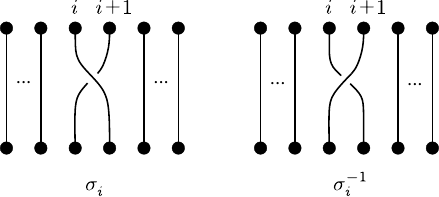}
    \caption{The generators of the braid group.}
    \label{fig:braid_group_generators}
\end{figure}

With each braid one can associate its permutation: this permutation takes
an element $k$ to $m$ if the strand starting with the $k$th upper point ends at
the $m$th lower point.

\begin{definition}
A braid is said to be {\em pure} if its permutation is identical. Obviously, pure braids generate a subgroup $\text{PB}_n \subset \text{B}_n$.
\end{definition}

The generator of the pure braid group is shown on Fig. \ref{fig:pure_braid_group_generators} \cite{Manturov-Knots}.

\begin{figure}[h]
    \centering
    \includegraphics[width=0.35\textwidth]{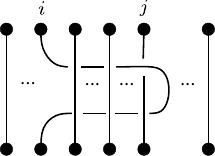}
    \caption{Generator $b_{ij}$ of the pure braid group.}
    \label{fig:pure_braid_group_generators}
\end{figure}

\begin{definition}
\label{def:pure_braid_representation}
The $n$-strand pure braid group $\text{PB}_n$ is the group given by the presentation with generators $b_{ij}$ and the following relations \cite{Staic}:
\begin{enumerate}
  \item $b_{ij}b_{kl}=b_{kl}b_{ij}$ for $k<l<i<j$ and $i<k<l<j$
  \item $b_{ij}b_{ik}b_{jk}=b_{jk}b_{ij}b_{ik}=b_{ik}b_{jk}b_{ij}$ for $i<j<k$
  \item $b_{jl}b_{kl}b_{ik}b_{jk}=b_{kl}b_{ik}b_{jk}b_{jl}$ for $i<j<k<l$.
\end{enumerate}
\end{definition}

\subsection{Dynamical systems}

Given a topological space $\Sigma$, called the {\em configuration space}; the elements of $\Sigma$ will be referred to as {\em particles}. The topology on $\Sigma^n$ defines a natural topology on the space of all continuous mappings $[0,1]\rightarrow \Sigma^n$ ($n\in\mathbb{N}$).

The {\em space of admissible dynamical systems} $\mathcal{D}$ is an open subset in the space of all maps $[0,1]\rightarrow \Sigma^n$.

A {\em state} is an element $D$ of $\Sigma^n$. By a {\em dynamical system} of $\mathcal{D}$ we mean the ordered set of particles $D(t) \in \Sigma^n$ for $t \in [0, 1]$. Herewith, $D(0)$ and $D(1)$ are called the initial state and the terminal state \cite{Invariants-and-Pictures}.

In \cite{Invariants-and-Pictures, Manturov-Kim} pure braids are considered as dynamical systems whose initial and final states coincide, see Fig. \ref{fig:dynamics}.
\begin{figure}[h]
    \centering
    \includegraphics[width=0.3\textwidth]{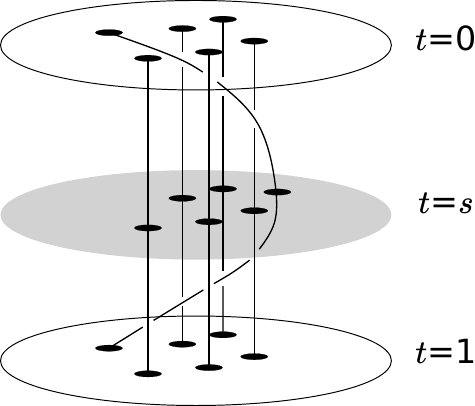}
    \caption{Dynamical system corresponding to $b_{ij} \in \text{PB}_n$ generator.}
    \label{fig:dynamics}
\end{figure}

\subsection{Delaunay triangulation}
Let us denote by $S$ a set of $n\ge 3$ point sites $p,q,r,\dots$ in the Euclidean plane, $\mathbb{R}^2$. For points $p=(p_1,p_2)$ and $x=(x_1,x_2)$, their Euclidean distance is given as
$$
d(p,x)=\sqrt{(p_1-x_1)^2+(p_2-x_2)^2}.
$$

The straight-line segment that connects two points $p$ and $q$ will be written
as $\overline{pq}$ in this section.

For $p,q \in S$, let $B(p,q)$ be the {\em bisector} of $p$ and $q$ (also called their
{\em separator}), which is the locus of all points in $\mathbb{R}^2$ at equal distance from
both $p$ and $q$. $B(p,q)$ is the perpendicular line through the midpoint of the
line segment $\overline{pq}$. It separates the halfplane
$$
D(p,q) = \{x\ |\ d(p, x) \le d(q, x)\}
$$
closer to $p$ from the halfplane $D(q,p)$ closer to $q$.

The {\em Vorono\"i region} of $p$ among the given set $S$ of sites, for short $\text{VR}(p,S)$, is the intersection of the $n-1$ halfplanes $D(p,q)$, where $q$ ranges over all the other sites in $S$:
$$
\text{VR}(p,S)=\bigcap_{q\in S,q\neq p}D(p,q).
$$

$\text{VR}(p,S)$ consists of all points $x\in \mathbb{R}^2$ for which $p$ is a nearest neighbor site.

The common boundary part of two Vorono\"i regions is called a {\em Vorono\"i edge}, if it contains more than one point.
\begin{definition}
The {\em Vorono\"i diagram} of $S$, for short $V(S)$, is defined as the union of all Vorono\"i edges.
\end{definition}

Endpoints of Vorono\"i edges are called {\em Vorono\"i vertices}; they belong to the common boundary of three or more Vorono\"i regions.

If a Vorono\"i edge $e$ borders the regions of $p$ and $q$ then $e\subset B(p,q)$ holds. That is, $V(S)$ is a {\em planar straight-line graph} whose edges emanate from Vorono\"i vertices.

There is an intuitive way of looking at the Vorono\"i diagram $V(S)$. Let
$x$ be an arbitrary point in the plane. We center a circle, $C$, at $x$ and let
its radius grow, from $0$ on. At some stage the expanding circle will, for the
first time, hit one or more sites of $S$. Now there are three different cases. If the circle $C$ expanding from point $x$ hits exactly one site, $p$, then $x$ belongs to the interior of region $\text{VR}(p,S)$. If $C$ hits exactly two sites, $p$ and $q$, then $x$ is an interior point of a Vorono\"i edge separating the Vorono\"i diagram  regions of $p$ and $q$. If $C$ hits three or more sites simultaneously, then $x$ is a Vorono\"i vertex adjacent to those regions whose sites have been hit \cite{Aurenhammer-Klein-Lee}.

In the book \cite{Aurenhammer-Klein-Lee} the definitions of {\em general position} of points means that no three of them are collinear, and no four of them are cocircular. In this part of the section we will use this concept, but for convenience this definition will be redefined in Sec. \ref{sec:triangles_number}.

The {\em Delaunay tessellation} $\text{DT}(S)$ is obtained by connecting with a line segment any two points $p$, $q$ of $S$, for which a circle exists that passes through $p$ and $q$ but does not contain any other site of $S$ in its interior or boundary. The edges of $\text{DT}(S)$ are called {\em Delaunay edges}.

If the point set $S$ is in general position then the dual graph $\text{DT}(S)$ of the Vorono\"i diagram $V(S)$, is a triangulation of $S$, called the {\em Delaunay triangulation}. Three points of $S$ give rise to a Delaunay triangle exactly if the circle they define does not enclose any other point of $S$ \cite{Aurenhammer-Klein-Lee}.

An example of both the Vorono\"i diagram and Delaunay triangulation are depicted in Fig. \ref{fig:voronoi_vs_delaunay}.

\begin{figure}[h]
    \centering
    \includegraphics[width=0.4\textwidth]{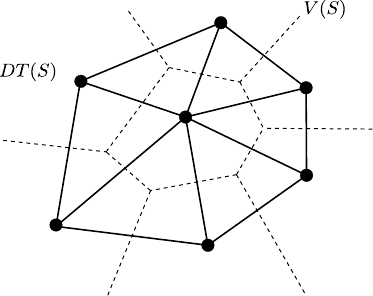}
    \caption{Vorono\"i diagram and Delaunay triangulation.}
    \label{fig:voronoi_vs_delaunay}
\end{figure}

\subsection{Number of triangles in a Delaunay triangulation}
\label{sec:triangles_number}
Consider the dynamic system $D(t)$, $0\le t\le 1$ of $n$ moving points on the plane. Let us numbers these points from $1$ to $n$.

Let us add three additional fixed points $a,b,c$ to the plane, so that at any moment $t$ the moving points of the dynamic system $D(t)$ do not go beyond the triangle $abc$ and the points never touch the edges of this triangle, see Fig. \ref{fig:delaunay_triangles}.

Thus, the total number of points on the plane will be equal to $m=n+3$.

Described contruction allows us to slightly simplify the definition of the general position of the points.

\begin{definition}
Further in the paper by the {\em general position} of $m$ points we will mean such a position of these points in which  no four of them are located on the same circle such that other points are at a futher distance from the center of this circle.\footnote{As we can see, this definition has a broad meaning and can depend on the context. Examples can be found in books \cite{Aurenhammer-Klein-Lee, Invariants-and-Pictures}.}
\end{definition}

\begin{figure}[h]
    \centering
    \includegraphics[width=0.8\textwidth]{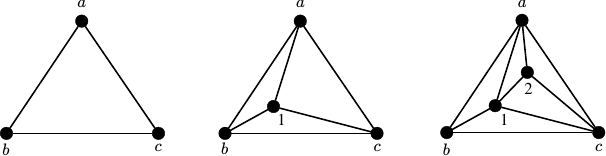}
    \caption{Delaunay triangulation inside triangle $abc$ for cases when $n=0$, $n=1$, $n=2$ respectively.}
    \label{fig:delaunay_triangles}
\end{figure}

It is obvious that if $m$ points are in general position, then the Delaunay triangulation in the construction above will be a trivalent graph.

\begin{lemma}
If all $m$ points are in general position, then the Delaunay triangulation on $m$ points has $2m-4$ faces in the construction described above.
\end{lemma}
\begin{proof}
Since we consider only those moments of time when the points are in a general position, then at these moments the Delaunay triangulation will be a trivalent graph. For a connected planar graph, the next Euler's formula holds for the numbers $v,e,f$ of vertices, edges and faces respectively:
$$
v-e+f=2.
$$

Let us express the number of edges in terms of the number of faces for the triangulation. Since each face contains exactly $3$ edges and each edge connects two faces consequently, it follows that: $v - \frac{3}{2}f + f = 2$. Finally, we obtain:
$$
f = 2v - 4.
$$

\end{proof}

Since in the construction above we used three additional points in addition to the $n$ points in the plane, the number of triangles can be expressed as $2n+2$.

Further in this paper, we will not take into account the outer edge and will assume that the number of triangles in the construction above is always equal to $2n+1$.

\subsection{Flips and pure braid group representation}
As the points of dynamical system $D(t)$ move, they may not be in general position at some moment. In this case, the Delaunay triangulation undergoes a flip, see Fig. \ref{fig:a_flip}. In the following paper, without limitation of generality, we will denote a flip by the expression $ik \to jl$ if this flip changes the diagonal $ik$ to the diagonal $jl$ inside the quadrilateral $ijkl$.

\begin{figure}[h]
    \centering
    \includegraphics[width=0.8\textwidth]{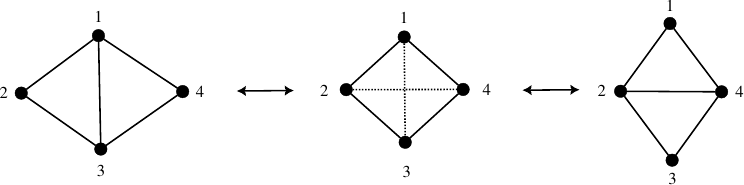}
    \caption{Flip $13 \to 24$.}
    \label{fig:a_flip}
\end{figure}

In \cite{Manturov-Kim, Manturov-Nikonov} both $\Gamma_n^4$ group and map $f_n:\text{PB}_n \rightarrow \Gamma_n^4$ were described. The group $\Gamma_n^4$ is naturally related to triangulations of $2$-surfaces and flips. A \mbox{generator} $d_{(ijkl)}$ of $\Gamma_n^4$ corresponds to a flip in a rectangle $ijkl$ of the Delaunay triangulation. For example, the generator $d_{(1234)}$ corresponds to the flip in Fig. \ref{fig:a_flip}.

\begin{definition}
The group $\Gamma_n^4$ is the group generated by
$$
\{d_{(ijkl)} | {i, j, k, l} \subset \bar{n}, |\{i, j, k, l\}| = 4\}
$$
subject to the following relations:
\begin{itemize}
    \item $d^2_{(ijkl)} = 1$ for $(i, j, k, l) \subset \bar{n}$,
    \item $d_{(ijkl)} d_{(stuv)} = d_{(stuv)} d_{(ijkl)}$, for $|\{i,j,k,l\} \cap \{s,t,u,v\}| < 3$,
    \item $d_{(ijkl)} d_{(ijlm)} d_{(jklm)} d_{(ijkm)} d_{(iklm)} = 1$ for distinct $i, j, k, l, m$.
    \item $d_{(ijkl)} = d_{(kjil)} = d_{(ilkj)} = d_{(klij)} = d_{(jkli)} = d_{(jilk)} = d_{(lkji)} = d_{(lijk)}$ for distinct $i,j,k,l$,
\end{itemize}
where $\bar{n}$ is the set $\{1,\dots,n\}$, $n\in\mathbb{N}$.
\end{definition}

The $f_n:\text{PB}_n \rightarrow \Gamma_n^4$ map is constructed as follows. The braid is considered as a dynamic system of moving points on a plane. On these points, a Delaunay triangulation is defined. This dynamic system corresponds to a finite number of critical moments of time at which the restructuring of the Delaunay triangulation occurs, i.e. when flips occur. The braid is associated with a sequence of flips. By associating the flips with the $\Gamma_n^4$ group generators, we obtain a word that corresponds to the braid. The relations in the $\Gamma_n^4$ group are composed in such a way that the isotopic braids correspond to the same words from the $\Gamma_n^4$ group. The $f_n$ map gives us the respresentation of the $\text{PB}_n$.

It was shown in \cite{Manturov-Zheyan} that each $d_{(ijkl)}$ generator can be associated with $2\times 2$ matrix of the form:
\begin{equation}
    \begin{pmatrix}
      \frac{\zeta_k-\zeta_j}{\zeta_l-\zeta_j} & \frac{\zeta_i-\zeta_j}{\zeta_l-\zeta_j} \\
      \frac{\zeta_k-\zeta_l}{\zeta_j-\zeta_l} & \frac{\zeta_i-\zeta_l}{\zeta_j-\zeta_l} \\
    \end{pmatrix} .
  \label{equation:flip-matrix}
\end{equation}

These matrices were obtained using the photography method. This method implies that we split the manifold into states and associate these states with some data. In this case, the state is the Delaunay triangulation, and the data are the area of the quadrilateral in which the flip occurs. The area of the triangle remains unchanged during the flip. This fact allows to obtain these matrices.

The sum of the elements in each column of this matrix is $1$. This allows to associate the matrix for the flip inside the pentagon as follows. Each triangulation of the pentagon consists of three triangles, each flip of the diagonals changes two of the three triangles and fixes the other one. The $3\times 3$ matrix is obtained from the $2\times 2$ matrix constructed for $d_{(ijkl)}$ by adding a diagonal entry $1$ corresponding to the fixed triangle. The Fig. \ref{fig:pentagon_flip} shows a flip of a pentagon. With such a reconstruction of the pentagon triangulation, one can associate a matrix of the form
\begin{equation}
  \begin{pmatrix}
      \frac{\zeta_k-\zeta_j}{\zeta_l-\zeta_j} & \frac{\zeta_i-\zeta_j}{\zeta_l-\zeta_j} & 0 \\
      \frac{\zeta_k-\zeta_l}{\zeta_j-\zeta_l} & \frac{\zeta_i-\zeta_l}{\zeta_j-\zeta_l} & 0 \\
      0 & 0 & 1 \\
  \end{pmatrix} .
  \label{equation:matrix}
\end{equation}

\begin{figure}[h]
    \centering
    \includegraphics[width=0.6\textwidth]{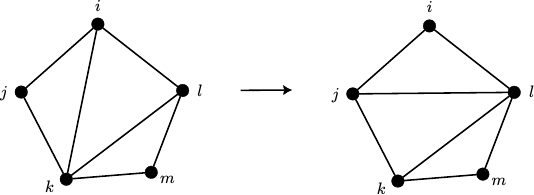}
    \caption{A pentagon flip.}
    \label{fig:pentagon_flip}
\end{figure}

As noted in the same paper, all generators of $\Gamma_n^4$ are involutions while these matrices are not. So, formally, this is not a representation of $\Gamma_n^4$, though it does give rise to an action of the pure braid group on labelled triangulations.

In this paper we will use similar ideas, but we will be able to construct matrices that will allow us to obtain a representation of the $\text{PB}_n$.

\section{Convention}
$\mathbb{Q}$ is the field of rational numbers.

In the rest of the paper we will consider $n$ moving points on a plane. Unless otherwise stated, we will assume that there are always three additional fixed points in the plane that never move, and that all other $n$ points are located inside the triangle formed by these three fixed points, and that moving points never touch the edges of this triangle, so the number of triangles in the Delaunay triangulation built on all these points always remains constant and it is equal to $(n+3)$.

\section{($2n+1) \times (2n+1$) Matrices}
Consider $n$ moving points in the plane. Let's associate each point on the plane with an unique element $\zeta_i \in \mathbb{Q}$ and denote the points by $\zeta_1,...,\zeta_{n+3}$ respectively. Assume these points are in general position over a period of time $t$. In this case, the Vorono\"i diagram is a graph of valency 3. Let $\mathcal{T}_t$ be the Delaunay triangulation corresponding to the Vorono\"i diagram during $t$. Let the set $\mathbf{T}_t$ be the set of triangles of the triangulation $\mathcal{T}_t$. 

We define a free $\mathbb{Q}$-module with basis $\mathbf{T}_t$ and denote it by $\mathbb{Q}^{(\mathbf{T}_t)}$. This module is a $(2n+1)$-dimensional vector space over $\mathbb{Q}$. We denote by $\mathbf{f}_t$ the ordered set of basis vectors of the $\mathbb{Q}^{(\mathbf{T}_t)}$ as follows:
$$
\mathbf{f}_{t} = (\Delta_{123}, \dots, \Delta_{{(n+1)}{(n+2)}{(n+3)}} ),
$$

\begin{remark}
The specific order of $\mathbb{Q}^{(\mathbf{T}_t)}$ basis can be fixed in any convenient way. Futher in this paper, we assume that indices of a basis vector are always ordered in ascending order, i.e. $\Delta_{ijk}$ means that $i<j<k$. We also assume that the basis $\mathbf{f}_{t}$ always complies lexicographic arrangement of triangles, for example $\Delta_{ijk}$ goes before $\Delta_{i'j'k'}$ if either $i<i'$ or $i=i'$ and $j<j'$ or $i=i'$, $j=j'$ and $k<k'$.
\end{remark}

As above, the Delaunay triangulation undergoes a flip at some critical moment when moving points are not in general position. Let the triangulation $\mathcal{T}_p$ differs from the triangulation $\mathcal{T}_t$ by a flip which changes the edge $\zeta_i\zeta_k$ to the edge $\zeta_j\zeta_l$. Then $\mathbf{T}_p\setminus \mathbf{T}_t=\{\Delta_{ijl}, \Delta_{jkl}\}$ and $\mathbf{T}_t\setminus \mathbf{T}_p=\{\Delta_{ijk}, \Delta_{ikl}\}$.

Now we will define a linear mapping $\gamma=\gamma_{\mathbb{Q}^{(\mathbf{T}_t)},\mathbb{Q}^{(\mathbf{T}_p)}}$. It can be uniquely determined by the images of the basis vectors of a vector space $\mathbb{Q}^{(\mathbf{T}_t)}$ \cite{Vinberg}. We define $\gamma$ as follows:
\begin{gather}
\gamma(\Delta)=\Delta \mbox{ for any } \Delta\in \mathbf{T}_t\cap \mathbf{T}_p, \label{eq:operator-1}\\
\gamma(\Delta_{ijk})=\frac{\zeta_{i} - \zeta_{l}}{\zeta_{i} - \zeta_{k}}\Delta_{ijl} + \frac{\zeta_{l} - \zeta_{k}}{\zeta_{i} - \zeta_{k}}\Delta_{jkl} \label{eq:operator-2}\\
\gamma(\Delta_{ikl})=\frac{\zeta_{i} - \zeta_{j}}{\zeta_{i} - \zeta_{k}}\Delta_{ijl} + \frac{\zeta_{j} - \zeta_{k}}{\zeta_{i} - \zeta_{k}}\Delta_{jkl} \label{eq:operator-3}
\end{gather}

Let $A_{\mathbf f_t,\mathbf f_p}$ be the matrix of the mapping $\gamma_{\mathbb{Q}^{(\mathbf{T}_t)},\mathbb{Q}^{(\mathbf{T}_p)}}$ in the bases $\mathbf f_t,\mathbf f_p$. If the triangulation $\mathcal{T}_t$ is clear from the context, then the triangulation $\mathcal{T}_p$ is determined by the flip $ik\to jl$, and we will use the notation $A_{ikjl}$ for the matrix $A_{\mathbf f_t,\mathbf f_p}$. This matrix has the form:

\begin{equation}
  A_{ikjl} = 
    \begin{pmatrix}
      * & ... & * & ... & * & ... & * \\
      \  & \  & \  & \vdots & \  & \  & \   \\
      * & ... & \frac{\zeta_i-\zeta_l}{\zeta_i-\zeta_k} & ... & \frac{\zeta_i-\zeta_j}{\zeta_i-\zeta_k} & ... & * \\
      \  & \  & \  & \vdots & \  & \  & \   \\
      * & ... & \frac{\zeta_l-\zeta_k}{\zeta_i-\zeta_k} & ... & \frac{\zeta_j-\zeta_k}{\zeta_i-\zeta_k} & ... & * \\
      \  & \  & \  & \vdots & \  & \  & \  \\
      * & ... & * & ... & * & ... & * \\
    \end{pmatrix}
  \label{equation:matrix_n_2}
\end{equation}
where $i,j,k,l \in \{1,2,...,n+3\}$; the {\em asterisk} is either $0$ or $1$ depending on the order of the vectors in the basis and is determined from formulas (\ref{eq:operator-1}) -- (\ref{eq:operator-3}).

Matrices with all other possible indices are defined similarly.
\begin{remark}
The sum of the column elements of these matrices is always equal to $1$.
\end{remark}

\begin{remark}
The $(2\times 2)$ matrix of the form $\begin{pmatrix} \frac{\zeta_i-\zeta_l}{\zeta_i-\zeta_k} & \frac{\zeta_i-\zeta_j}{\zeta_i-\zeta_k} \\ \frac{\zeta_l-\zeta_k}{\zeta_i-\zeta_k} & \frac{\zeta_j-\zeta_k}{\zeta_i-\zeta_k} \end{pmatrix}^{-1}$ is equal to the matrix from formula (\ref{equation:flip-matrix}).
\end{remark}

\begin{remark}
The set of all matrices of the form (\ref{equation:matrix_n_2}) is a subgroup of $\text{GL}_{2n+1}(\mathbb{Q})$. Indeed, it can easily be checked that these matrices are non-singular.
\end{remark}

\begin{remark}
Matrices of the form (\ref{equation:matrix_n_2}) are a special case of the matrices that were considered in the paper \cite{Korepanov-pentagon} in p. $4$.
\end{remark}

\begin{lemma}
The matrices of the form (\ref{equation:matrix_n_2}) satisfy the following relations:

\begin{enumerate}[label=(\alph*)]
    \item $A_{ikjl} = A^{-1}_{jlik}$, where $i,j,k,l\in\{1,\dots,n+3\}$, $A_{ikjl}=A_{\mathbf f_t,\mathbf f_p}$, $A_{jlik}=A_{\mathbf f_p,\mathbf f_t}$, $|\mathbf{T}_t \setminus \mathbf{T}_p|=2$,\label{relation:inverse}
    \item $A_{suvw}A_{ikjl}=A_{ikjl}'A_{suvw}'$, where $|\{i,j,k,l\}\cap\{s,u,v,w\}| < 3$, \\ $i,j,k,l,s,u,v,w\in\{1,\dots,n+3\}$, $A_{ikjl}=A_{\mathbf f_t,\mathbf f_p}$, $A_{suvw}=A_{\mathbf f_p,\mathbf f_q}$, $A_{suvw}'=A_{\mathbf f_t,\mathbf f_r}$, $A_{ikjl}'=A_{\mathbf f_r,\mathbf f_q}$, $|\mathbf{T}_t \setminus \mathbf{T}_p|=|\mathbf{T}_p \setminus \mathbf{T}_q|=|\mathbf{T}_t \setminus \mathbf{T}_{r}|=|\mathbf{T}_{r} \setminus \mathbf{T}_q|=2$,
    \label{relation:commutativity}
    \item $A_{jlik}A_{jmil}A_{kmjl}A_{ikjm}A_{ilkm}=I_{2n+1}$, where $i,j,k,l,m\in\{1,\dots,n+3\}$ are five points which are located on the same circle such that other points are at a futher distance from the center of this circle, \mbox{Fig. \ref{fig:ijklm-flip}}; $A_{ilkm}=A_{\mathbf{T}_t,\mathbf{T}_p}$, $A_{ikjm}=A_{\mathbf{T}_p,\mathbf{T}_q}$, $A_{kmjl}=A_{\mathbf{T}_q,\mathbf{T}_r}$, $A_{jmil}=A_{\mathbf{T}_r,\mathbf{T}_s}$, $A_{jlik}=A_{\mathbf{T}_s,\mathbf{T}_t}$; $|\mathbf{T}_t \setminus \mathbf{T}_p|=|\mathbf{T}_p \setminus \mathbf{T}_q|=|\mathbf{T}_q \setminus \mathbf{T}_r|=|\mathbf{T}_r \setminus \mathbf{T}_s|=|\mathbf{T}_s \setminus \mathbf{T}_t|=2$; $I_{2n+1}$ is the identity matrix of size $(2n+1)\times (2n+1)$.
    \label{relation:pentagon}
\end{enumerate}
\end{lemma}

The relation \ref{relation:commutativity} is called a \textit{far commutativity relation}. The relation \ref{relation:pentagon} is called a \textit{pentagon relation}.

Let us prove the first statement. Let us construct two new $2\times 2$ matrices $a_{ikjl},a_{jlik}$ based on the matrices in relation \ref{relation:inverse} as follows:

\begin{equation}\label{equation:A_ikjl}
a_{ikjl}=\begin{pmatrix} \frac{\zeta_i-\zeta_l}{\zeta_i-\zeta_k} & \frac{\zeta_i-\zeta_j}{\zeta_i-\zeta_k} \\ \frac{\zeta_l-\zeta_k}{\zeta_i-\zeta_k} & \frac{\zeta_j-\zeta_k}{\zeta_i-\zeta_k} \end{pmatrix}
\end{equation}
\begin{equation}\label{equation:A_jlik}
a_{jlik}=\begin{pmatrix} \frac{\zeta_j-\zeta_k}{\zeta_j-\zeta_l} & \frac{\zeta_j-\zeta_i}{\zeta_j-\zeta_l} \\ \frac{\zeta_k-\zeta_l}{\zeta_j-\zeta_l} & \frac{\zeta_i-\zeta_l}{\zeta_j-\zeta_l} \end{pmatrix}
\end{equation}

Next, we will use the inverse matrix formula described in the paper \cite{Manturov-Zheyan}. If $
a_{ikjl}=
  \begin{pmatrix}
    a & b \\
    c & d \\
  \end{pmatrix}
$, then
$a_{ikjl}^{-1}=
  \begin{pmatrix}
    \frac{1-b}{a-b} & \frac{-b}{a-b} \\
    \frac{a-1}{a-b} & \frac{a}{a-b}  \\
  \end{pmatrix}
$. Let us substitute the value of the elements from formula (\ref{equation:A_ikjl}) into the $a_{ikjl}^{-1}$ matrix formula. For the first element we get:
$$
\frac{1-b}{a-b}=\frac{1-\frac{\zeta_i-\zeta_j}{\zeta_i-\zeta_k}}{\frac{\zeta_i-\zeta_l}{\zeta_i-\zeta_k} - \frac{\zeta_i-\zeta_j}{\zeta_i-\zeta_k}}=\frac{\frac{\zeta_i-\zeta_k-\zeta_i+\zeta_j}{\zeta_i-\zeta_k}}{\frac{\zeta_i-\zeta_l-\zeta_i+\zeta_j}{\zeta_i-\zeta_k}}=\frac{\zeta_j-\zeta_k}{\zeta_j-\zeta_l}.
$$

It is obvious that the obtained element exactly coincides with the first element of the matrix $a_{jlik}$ from the formula (\ref{equation:A_jlik}). All other elements are proven similarly. The reasoning above will be true if we expand the matrices from formulas (\ref{equation:A_ikjl}), (\ref{equation:A_jlik}) with unit blocks.

The proof of commutativity of matrices follows from the fact that matrix elements which are neither $0$ nor $1$ in these matrices do not intersect. Therefore, when multiplying such matrices, each such element will be multiplied by either $0$ or $1$ of the second matrix.

Let us give an explicit example for matrices of size $6\times 6$. Let us assume that there are six points on the plane, points $\zeta_1,\zeta_2,\zeta_3$ are static, and points $\zeta_4,\zeta_5,\zeta_6$ are moving, see Fig. \ref{fig:commute_example}.
\begin{figure}[h]
    \centering
    \includegraphics[width=0.6\textwidth]{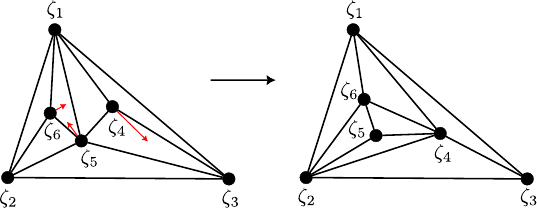}
    \caption{Example of two flips; arrows show the directions of movement of the points.}
    \label{fig:commute_example}
\end{figure}
Let us assume that during the movement two flips occurred, namely, edge $\zeta_3\zeta_5$ changed to edge $\zeta_2\zeta_4$, and edge $\zeta_1\zeta_5$ changed to edge $\zeta_4\zeta_6$. Such an event will correspond to the product of matrices of the form:
\bigbreak
\resizebox{1.0\hsize}{!}{$
\left(\begin{array}{rrrrrrr}
1 & 0 & 0 & 0 & 0 & 0 & 0 \\
0 & 1 & 0 & 0 & 0 & 0 & 0 \\
0 & 0 & \frac{\zeta_{1} - \zeta_{6}}{\zeta_{1} - \zeta_{5}} & \frac{\zeta_{1} - \zeta_{4}}{\zeta_{1} - \zeta_{5}} & 0 & 0 & 0 \\
0 & 0 & 0 & 0 & 1 & 0 & 0 \\
0 & 0 & 0 & 0 & 0 & 1 & 0 \\
0 & 0 & 0 & 0 & 0 & 0 & 1 \\
0 & 0 & \frac{\zeta_{6} - \zeta_{5}}{\zeta_{1} - \zeta_{5}} & \frac{\zeta_{4} - \zeta_{5}}{\zeta_{1} - \zeta_{5}} & 0 & 0 & 0
\end{array}\right)
\left(\begin{array}{rrrrrrr}
1 & 0 & 0 & 0 & 0 & 0 & 0 \\
0 & 1 & 0 & 0 & 0 & 0 & 0 \\
0 & 0 & 1 & 0 & 0 & 0 & 0 \\
0 & 0 & 0 & 1 & 0 & 0 & 0 \\
0 & 0 & 0 & 0 & \frac{\zeta_{3} - \zeta_{4}}{\zeta_{3} - \zeta_{5}} & 0 & -\frac{\zeta_{2} - \zeta_{3}}{\zeta_{3} - \zeta_{5}} \\
0 & 0 & 0 & 0 & \frac{\zeta_{4} - \zeta_{5}}{\zeta_{3} - \zeta_{5}} & 0 & \frac{\zeta_{2} - \zeta_{5}}{\zeta_{3} - \zeta_{5}} \\
0 & 0 & 0 & 0 & 0 & 1 & 0
\end{array}\right)=
\left(\begin{array}{rrrrrrr}
1 & 0 & 0 & 0 & 0 & 0 & 0 \\
0 & 1 & 0 & 0 & 0 & 0 & 0 \\
0 & 0 & \frac{\zeta_{1} - \zeta_{6}}{\zeta_{1} - \zeta_{5}} & \frac{\zeta_{1} - \zeta_{4}}{\zeta_{1} - \zeta_{5}} & 0 & 0 & 0 \\
0 & 0 & 0 & 0 & \frac{\zeta_{3} - \zeta_{4}}{\zeta_{3} - \zeta_{5}} & 0 & \frac{\zeta_{3} - \zeta_{2}}{\zeta_{3} - \zeta_{5}} \\
0 & 0 & 0 & 0 & \frac{\zeta_{4} - \zeta_{5}}{\zeta_{3} - \zeta_{5}} & 0 & \frac{\zeta_{2} - \zeta_{5}}{\zeta_{3} - \zeta_{5}} \\
0 & 0 & 0 & 0 & 0 & 1 & 0 \\
0 & 0 & \frac{\zeta_{6} - \zeta_{5}}{\zeta_{1} - \zeta_{5}} & \frac{\zeta_{4} - \zeta_{5}}{\zeta_{1} - \zeta_{5}} & 0 & 0 & 0
\end{array}\right)
$}

\bigbreak
Conversely, if edge $\zeta_1\zeta_5$ is first replaced by edge $\zeta_4\zeta_6$, and then edge $\zeta_3\zeta_5$ is replaced by $\zeta_2\zeta_4$, we will get a product of the form:

\bigbreak
\resizebox{1.0\hsize}{!}{$
\left(\begin{array}{rrrrrrr}
1 & 0 & 0 & 0 & 0 & 0 & 0 \\
0 & 1 & 0 & 0 & 0 & 0 & 0 \\
0 & 0 & 1 & 0 & 0 & 0 & 0 \\
0 & 0 & 0 & \frac{\zeta_{3} - \zeta_{4}}{\zeta_{3} - \zeta_{5}} & 0 & \frac{\zeta_{3} - \zeta_{2}}{\zeta_{3} - \zeta_{5}} & 0 \\
0 & 0 & 0 & \frac{\zeta_{4} - \zeta_{5}}{\zeta_{3} - \zeta_{5}} & 0 & \frac{\zeta_{2} - \zeta_{5}}{\zeta_{3} - \zeta_{5}} & 0 \\
0 & 0 & 0 & 0 & 1 & 0 & 0 \\
0 & 0 & 0 & 0 & 0 & 0 & 1
\end{array}\right)
\left(\begin{array}{rrrrrrr}
1 & 0 & 0 & 0 & 0 & 0 & 0 \\
0 & 1 & 0 & 0 & 0 & 0 & 0 \\
0 & 0 & \frac{\zeta_{1} - \zeta_{6}}{\zeta_{1} - \zeta_{5}} & \frac{\zeta_{1} - \zeta_{4}}{\zeta_{1} - \zeta_{5}} & 0 & 0 & 0 \\
0 & 0 & 0 & 0 & 1 & 0 & 0 \\
0 & 0 & 0 & 0 & 0 & 1 & 0 \\
0 & 0 & 0 & 0 & 0 & 0 & 1 \\
0 & 0 & \frac{\zeta_{6} - \zeta_{5}}{\zeta_{1} - \zeta_{5}} & \frac{\zeta_{4} - \zeta_{5}}{\zeta_{1} - \zeta_{5}} & 0 & 0 & 0
\end{array}\right)=
\left(\begin{array}{rrrrrrr}
1 & 0 & 0 & 0 & 0 & 0 & 0 \\
0 & 1 & 0 & 0 & 0 & 0 & 0 \\
0 & 0 & \frac{\zeta_{1} - \zeta_{6}}{\zeta_{1} - \zeta_{5}} & \frac{\zeta_{1} - \zeta_{4}}{\zeta_{1} - \zeta_{5}} & 0 & 0 & 0 \\
0 & 0 & 0 & 0 & \frac{\zeta_{3} - \zeta_{4}}{\zeta_{3} - \zeta_{5}} & 0 & \frac{\zeta_{3} - \zeta_{2}}{\zeta_{3} - \zeta_{5}} \\
0 & 0 & 0 & 0 & \frac{\zeta_{4} - \zeta_{5}}{\zeta_{3} - \zeta_{5}} & 0 & \frac{\zeta_{2} - \zeta_{5}}{\zeta_{3} - \zeta_{5}} \\
0 & 0 & 0 & 0 & 0 & 1 & 0 \\
0 & 0 & \frac{\zeta_{6} - \zeta_{5}}{\zeta_{1} - \zeta_{5}} & \frac{\zeta_{4} - \zeta_{5}}{\zeta_{1} - \zeta_{5}} & 0 & 0 & 0
\end{array}\right)
$}

\bigbreak
Obviously, we have obtained identical matrices. Therefore, these arguments will be true for matrices of larger size as well.

The pentagon relation for $3\times 3$ matrices is verified by direct computation, see. App. \ref{appendix:pentagon}. Larger matrices are obtained from $3 \times 3$ matrices by adding diagonal entries $1$ corresponding to the fixed triangles. The matrix blocks either do not change or the size of blocks changes the same for all matrices. The pentagon relation is preserved in this case.

\section{The main theorem}
\label{section:main_theorem}
The ($n+3$) points considered in the previous section form a dynamic system, which we denote by $D'$. Let us assume that the initial and final states of the $D'$ coincide. Then $D'$ corresponds to the element $b_{ij}$ of $n$-strand pure braid group described in Section \ref{section:braid}.

Let us enumerate the moments $0 < t_1 < t_2 < \dots < t_l < 1$ such that at the moment $t_k$ four points of the dynamic system $D'$ belong to one circle without points inside the circle.
Let us assume that at the moment $t_k$ the flip $pq\rightarrow rs$ occurred. Let $\mathcal{T}$ be the triangulation before the flip, and $\mathcal{T}'$ be the triangulation after the flip.
By $\mathbf{T}$ we denote a set of triangles that corresponds to triangulation $\mathcal{T}$ and by $\mathbf{T}'$ we denote a set of triangles that corresponds to triangulation $\mathcal{T}'$. 
As before, we define free $\mathbb{Q}$-modules on $\mathbf{T}$ and $\mathbf{T}'$ and denote them $\mathbb{Q}^{\mathbf{T}}$ and $\mathbb{Q}^{\mathbf{T}'}$, respectively. Let us denote by $\mathbf{f}$ the ordered basis of the module $\mathbb{Q}^{\mathbf{T}}$ and by $\mathbf{f}'$ the ordered basis of the module $\mathbb{Q}^{\mathbf{T}'}$. Now we associate the flip $pq\rightarrow rs$ with a linear mapping $\gamma_k:\mathbb{Q}^{\mathbf{T}}\rightarrow \mathbb{Q}^{\mathbf{T}'}$ by formulas (\ref{eq:operator-1}) -- (\ref{eq:operator-3}). By $A_k=A_{pqrs}$ we denote the matrix of the $\gamma_k$ and this matrix has a form (\ref{equation:matrix_n_2}). In matrix notation, this can be written as follows:
\begin{equation}
\gamma_k(\mathbf{f}) = \mathbf{f}'A_{pqrs}.
\label{eq:phi_1}
\end{equation}

Suppose that at time $t_{k+1}$ a flip $rs\rightarrow uv$ occurs. We obtain a new triangulation $\mathcal{T}''$ after the $t_{k+1}$ moment. The set of triangles $\mathbf{T}''$ corresponds to the triangulation $\mathcal{T}''$. We define a vector space $\mathbb{Q}^{\mathbf{T}''}$ with an ordered basis $\mathbf{f}''$ on the set $\mathbf{T}''$. We associate the flip $rs\rightarrow uv$ with the mapping $\gamma_{k+1}:\mathbb{Q}^{\mathbf{T}'}\rightarrow \mathbb{Q}^{\mathbf{T}''}$ by the same way. We denote the $\gamma_{k+1}$ matrix as $A_{k+1}=A_{rsuv}$. In matrix notation, this can be written as follows:
\begin{equation}
\gamma_{k+1}(\mathbf{f}') = \mathbf{f}''A_{rsuv}.
\label{eq:phi_2}
\end{equation}

From the axioms of linear mappings, it follows that we can build a composition of mappings $\gamma_{k}$ and $\gamma_{k+1}$ from formulas (\ref{eq:phi_1}) and (\ref{eq:phi_2}) as follows:
\begin{equation}
\gamma_{k+1}(\gamma_k(\mathbf{f})) = \mathbf{f}''A_{k+1}A_{k}.
\end{equation}

Thus, each moment $t_k$ can be associated with a linear mapping $\gamma_k$ and therefore the dynamic system $D'$ corresponding to a certain braid can be represented as a composition of such linear mappings. Since the state of the dynamic system $D'$ at moment $0$ and moment $1$ coincides, the bases of the corresponding vector spaces also coincide, and therefore we can associate an appropriate pure braid with a composition of linear mappings of the following form:
\begin{equation}
\gamma_{l}\gamma_{l-1}...\gamma_{2}\gamma_{1}(\mathbf{f})=\mathbf{f}A_lA_{l-1}...A_2A_1,
\end{equation}
where $\mathbf{f}$ is an ordered basis of the vector space corresponding to the triangulation defined at the ($n+3$) points of the dynamic system $D'$ at the initial state $D'(0)$.

Since after choosing the bases the linear mapping is completely determined by the linear mapping matrix, we can associate with the pure braid the product of the corresponding linear mappings matrices and prove the following theorem.

\begin{theorem}
The map $f_n:\text{PB}_n \rightarrow \textsc{GL}_{2n+1}(\mathbb{Q})$ is defined as follows:
$$
f_n(b_{ij}) = \prod_{k=1}^{l}A_{l-k+1},
$$
where the $A_{i}$ is a matrix of the form (\ref{equation:matrix_n_2}), is a well-defined homomorphism.

\end{theorem}
\begin{proof}
We consider isotopies between two pure $n$-strand braids as isotopies of general position state of ($n+3$) points on the plane. A finite number of events of codimension at most two correspond to the general position state isotopies. These events were classified in \cite{Fedoseev-Manturov-Nikonov}. Now we list them explicitly.
\begin{enumerate}[label=(\roman*)]
    \item One point moving on the plane is tangential to the circle, which passes through three points, see  Fig. \ref{fig:condimetion-dots}. This event corresponds to the \mbox{$A_{pvsu} \cdot A_{supv}$} product that is equivalent to the identity matrix by the relation \ref{relation:inverse}.

\begin{figure}[h]
    \centering
    \includegraphics[width=0.7\textwidth]{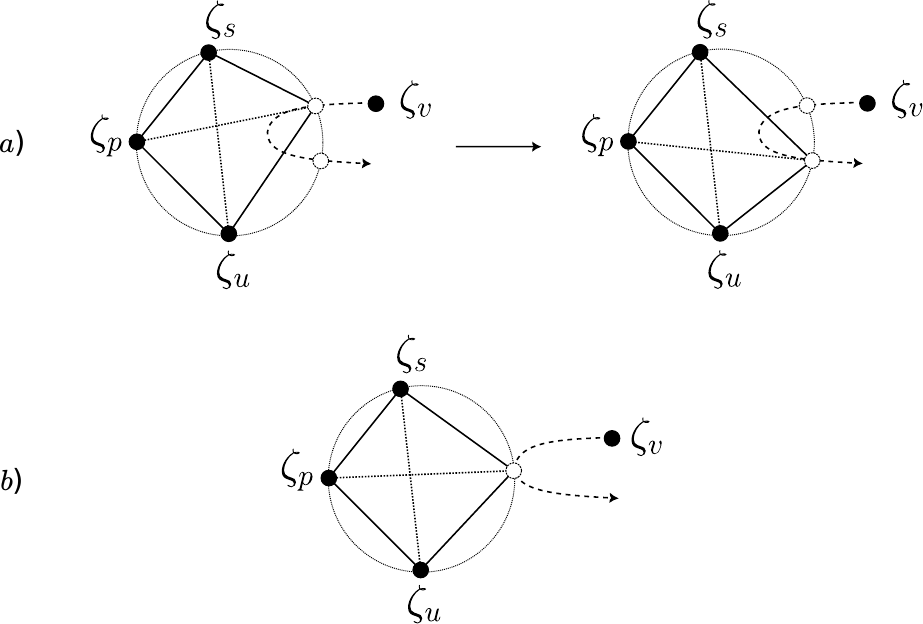}
    \caption{Codimension-$2$ events; point $\zeta_v$ moves, being tangential to the circle, which passes through $\zeta_s$, $\zeta_p$ and $\zeta_u$.}
    \label{fig:condimetion-dots}
\end{figure}
    \item There are two sets $A$ and $B$ of four points, which are on the same circles such that $|A \cap B| \le 2$, see Fig. \ref{fig:commutativity-dots}. Such an event corresponds to two flips in different quadrilaterials in an arbitrary order. For example, these can be flips $ju \to iv, ik \to jl$, or vice versa $ik \to jl, ju \to iv$. The product of the matrices corresponding to these flips must be equivalent regardless of the order of multiplication of these matrices which exactly corresponds to the far commutativity relation \ref{relation:commutativity}.
\begin{figure}[h]
    \centering
    \includegraphics[width=0.4\textwidth]{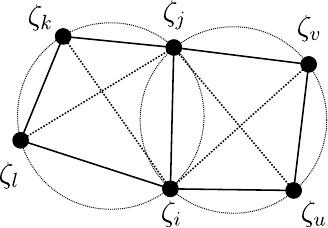}
    \caption{Codimension-$2$ event; two sets A and B of four points on the circles such that $|A \cap B| = 2$.}
    \label{fig:commutativity-dots}
\end{figure}
    \item There are five points $\{\zeta_i,\zeta_j,\zeta_k,\zeta_l,\zeta_m\}$ on the same circle; all other points are at a futher distance from the center of this circle. We obtain the sequence of five subsets of $\{\zeta_i,\zeta_j,\zeta_k,\zeta_l,\zeta_m\}$ with four points on the same circle, which corresponds to the flips on the pentagon, see Fig. \ref{fig:ijklm-flip}. This corresponds to the pentagon relation \ref{relation:pentagon}.
\begin{figure}[h]
    \centering
    \includegraphics[width=0.8\textwidth]{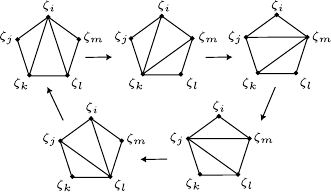}
    \caption{Codimension-$2$ event; a sequence of five subsets of $\{\zeta_i,\zeta_j,\zeta_k,\zeta_l,\zeta_m\}$, corresponding to flips on the pentagon.}
    \label{fig:ijklm-flip}
\end{figure}
    \label{item:pentagon}
\end{enumerate}
\end{proof}

Geometric interpretation of the \ref{item:pentagon} event from the point of view of the braids
group was described in \cite{Manturov-Nikonov}.

In this paper the ``photography method'' was also used which you can read briefly about in the paper \cite{Manturov-Zheyan}. This method helps to understand pentagon relation geometrically.

\section{Further directions}
Among further directions which may appear in further papers is knots invariants constructing.

A braid closure allows to construct a knot from the braid. Also it is possible to transform one braid to another by a sequence of Markov movements. The first move is the conjugation. The second move is adding a new strand to the braid and twisting this strand once with the last strand of the braid.

Having a braids invariant in matrices, we need to find some characteristics of the matrix that are invariant under the first and second Markov movements. It can be a trace or characteristic polynomial of the matrix. This will allow us to construct an invariant
for knots.

\section{Acknowledgements}
\thanks{I express my gratitude to my scientific supervisor Vasily O. Manturov and Igor M. Nikonov for valuable advice and assistance in writing this article and to Igor G. Korepanov for the original ideas that formed its basis. I also express my gratitude to Kim Seongjeong, Alexey V. Sleptsov, Wan Zheyan and Louis Kauffman for their comments.}

\pagebreak

\begin{appendices}
\section{Pentagon calculation example}
\label{appendix:pentagon}

Let us consider the composition of linear mappings that correspond to the reconstruction of the triangulation inside the pentagon as it shown in Fig. \ref{fig:ijklm-flip}.

We will denote the basis vectors by the indices of the variables, for example, the vector that corresponds to the triangle formed by the variables $\zeta_i, \zeta_j, \zeta_k$ will be denoted simply as $ijk$.

Thus, we can assume that a free $\mathbb{Q}$-module with a basis $(ijk,ikl,ilm)$ corresponds to the the initial state of the considered dynamic system.

The restructuring of the pentagon triangulation corresponds to the following series of flips: $\zeta_i \zeta_l \rightarrow \zeta_k \zeta_m$, $\zeta_i \zeta_k \rightarrow \zeta_j \zeta_m$, $\zeta_k \zeta_m \rightarrow \zeta_j \zeta_l$, $\zeta_j \zeta_m \rightarrow \zeta_i \zeta_l$, $\zeta_j \zeta_l \rightarrow \zeta_i \zeta_k$.

In turn, each of these flips is associated with the following linear mappings with the corresponding bases:

$$
\gamma_1(ijk,ikl,ilm)
=
(ijk,ikm,klm)
\left(\begin{array}{rrr}
1 & 0 & 0 \\
0 & \frac{\zeta_{i} - \zeta_{m}}{\zeta_{i} - \zeta_{l}} & \frac{\zeta_{i} - \zeta_{k}}{\zeta_{i} - \zeta_{l}} \\
0 & \frac{\zeta_{m} - \zeta_{l}}{\zeta_{i} - \zeta_{l}} & \frac{\zeta_{k} - \zeta_{l}}{\zeta_{i} - \zeta_{l}}
\end{array}\right),
$$

$$
\gamma_2(ijk,ikm,klm)
=
(ijm,jkm,klm)
\left(\begin{array}{rrr}
\frac{\zeta_{i} - \zeta_{m}}{\zeta_{i} - \zeta_{k}} & \frac{\zeta_{i} - \zeta_{j}}{\zeta_{i} - \zeta_{k}} & 0 \\
\frac{\zeta_{m} - \zeta_{k}}{\zeta_{i} - \zeta_{k}} & \frac{\zeta_{j} - \zeta_{k}}{\zeta_{i} - \zeta_{k}} & 0 \\
0 & 0 & 1
\end{array}\right),
$$

$$
\gamma_3(ijm,jkm,klm)
=
(ijm,jkl,jlm)
\left(\begin{array}{rrr}
1 & 0 & 0 \\
0 & \frac{\zeta_{k} - \zeta_{l}}{\zeta_{k} - \zeta_{m}} & \frac{\zeta_{k} - \zeta_{j}}{\zeta_{k} - \zeta_{m}} \\
0 & \frac{\zeta_{l} - \zeta_{m}}{\zeta_{k} - \zeta_{m}} & \frac{\zeta_{j} - \zeta_{m}}{\zeta_{k} - \zeta_{m}}
\end{array}\right),
$$

$$
\gamma_4(ijm,jkl,jlm)
=
(ijl,ilm,jkl)
\left(\begin{array}{rrr}
\frac{\zeta_{j} - \zeta_{l}}{\zeta_{j} - \zeta_{m}} & 0 & \frac{\zeta_{j} - \zeta_{i}}{\zeta_{j} - \zeta_{m}} \\
\frac{\zeta_{l} - \zeta_{m}}{\zeta_{j} - \zeta_{m}} & 0 & \frac{\zeta_{i} - \zeta_{m}}{\zeta_{j} - \zeta_{m}} \\
0 & 1 & 0
\end{array}\right),
$$

$$
\gamma_5(ijl,ilm,jkl)
=
(ijk,ikl,ilm)
\left(\begin{array}{rrr}
\frac{\zeta_{j} - \zeta_{k}}{\zeta_{j} - \zeta_{l}} & 0 & \frac{\zeta_{j} - \zeta_{i}}{\zeta_{j} - \zeta_{l}} \\
\frac{\zeta_{k} - \zeta_{l}}{\zeta_{j} - \zeta_{l}} & 0 & \frac{\zeta_{i} - \zeta_{l}}{\zeta_{j} - \zeta_{l}} \\
0 & 1 & 0
\end{array}\right).
$$

The composition of such linear operators will look like this:
$$
\gamma_5\gamma_4\gamma_3\gamma_2\gamma_1(ijk,ikl,ilm)=(ijk,ikl,ilm)B,
$$

where $B$ is a matrix that is equal to the product of the corresponding linear mappings matrices in the following order:
\begin{gather*}
B = 
\left(\begin{array}{rrr}
\frac{\zeta_{j} - \zeta_{k}}{\zeta_{j} - \zeta_{l}} & 0 & \frac{\zeta_{j} - \zeta_{i}}{\zeta_{j} - \zeta_{l}} \\
\frac{\zeta_{k} - \zeta_{l}}{\zeta_{j} - \zeta_{l}} & 0 & \frac{\zeta_{i} - \zeta_{l}}{\zeta_{j} - \zeta_{l}} \\
0 & 1 & 0
\end{array}\right)
\left(\begin{array}{rrr}
\frac{\zeta_{j} - \zeta_{l}}{\zeta_{j} - \zeta_{m}} & 0 & \frac{\zeta_{j} - \zeta_{i}}{\zeta_{j} - \zeta_{m}} \\
\frac{\zeta_{l} - \zeta_{m}}{\zeta_{j} - \zeta_{m}} & 0 & \frac{\zeta_{i} - \zeta_{m}}{\zeta_{j} - \zeta_{m}} \\
0 & 1 & 0
\end{array}\right)
\left(\begin{array}{rrr}
1 & 0 & 0 \\
0 & \frac{\zeta_{k} - \zeta_{l}}{\zeta_{k} - \zeta_{m}} & \frac{\zeta_{k} - \zeta_{j}}{\zeta_{k} - \zeta_{m}} \\
0 & \frac{\zeta_{l} - \zeta_{m}}{\zeta_{k} - \zeta_{m}} & \frac{\zeta_{j} - \zeta_{m}}{\zeta_{k} - \zeta_{m}}
\end{array}\right) \\
\cdot\left(\begin{array}{rrr}
\frac{\zeta_{i} - \zeta_{m}}{\zeta_{i} - \zeta_{k}} & \frac{\zeta_{i} - \zeta_{j}}{\zeta_{i} - \zeta_{k}} & 0 \\
\frac{\zeta_{m} - \zeta_{k}}{\zeta_{i} - \zeta_{k}} & \frac{\zeta_{j} - \zeta_{k}}{\zeta_{i} - \zeta_{k}} & 0 \\
0 & 0 & 1
\end{array}\right)
\left(\begin{array}{rrr}
1 & 0 & 0 \\
0 & \frac{\zeta_{i} - \zeta_{m}}{\zeta_{i} - \zeta_{l}} & \frac{\zeta_{i} - \zeta_{k}}{\zeta_{i} - \zeta_{l}} \\
0 & \frac{\zeta_{m} - \zeta_{l}}{\zeta_{i} - \zeta_{l}} & \frac{\zeta_{k} - \zeta_{l}}{\zeta_{i} - \zeta_{l}}
\end{array}\right) = I_3.
\end{gather*}

Direct computation on a computer allows us to verify that the product of such matrices is equal to the identity matrix $I_3$.

\section{Invariant calculation example}
Let us give an example of calculating the invariant. As an example, let us take the first relation from the presentation of the $\text{PB}_n$ given in the Definition \ref{def:pure_braid_representation}:
$$
b_{ij}b_{kl}=b_{kl}b_{ij}
$$
for $k<l<i<j$ and $i<k<l<j$.

An example of a dynamic system that corresponds to the braids $b_{ij},b_{kl}$ for the case $i<k<l<j$ is shown in the Fig. \ref{fig:pure_braid_calc_example}. According to this picture $\zeta_i=\zeta_4$, $\zeta_k=\zeta_5$, $\zeta_l=\zeta_7$, $\zeta_j=\zeta_8$.

\begin{figure}[h]
    \centering
    \includegraphics[width=0.8\textwidth]{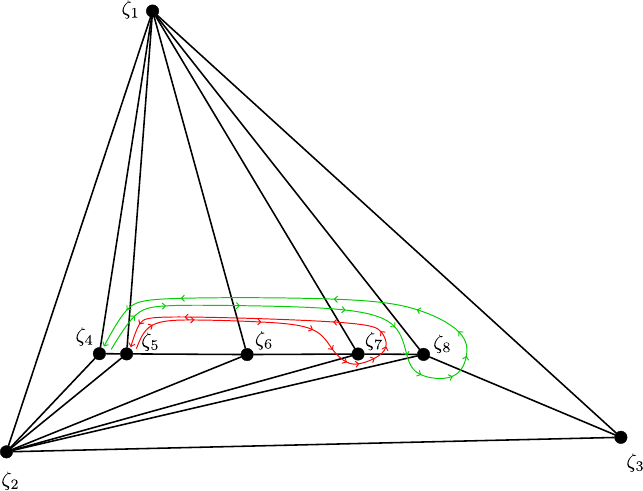}
    \caption{Example of a dynamic system of pure braids; arrows indicate the trajectories of the point $\zeta_i=\zeta_4$ of the braid $b_{ij}$ and the point $\zeta_k=\zeta_5$ of the braid $b_{kl}$.}
    \label{fig:pure_braid_calc_example}
\end{figure}

Further, for convenience, we will assume that $\zeta_i=i,\ i \in \{1,\dots,8\}$.

The braid $b_{ij}$ will correspond to the product of the following matrices:
\begin{gather*}
f_5(b_{ij})=A_{4615}A_{1524}A_{4716}A_{4817}A_{1645}A_{1746}A_{3418}A_{1847}A_{2438}A_{4728} \cdot\\ \cdot A_{3814}A_{1478}A_{2834}A_{7824}A_{4617}A_{4516}A_{1748}A_{1647}A_{2415}A_{1546}
\end{gather*}

Let us denote the product of these matrices by the  $A_{b_{ij}}$. In matrix form this product looks as follow:

$$
A_{b_{ij}}=f_5(b_{ij})=
$$

\begin{gather*}
\left(\begin{array}{rrrrrrrrrrr}
1 & 0 & 0 & 0 & 0 & 0 & 0 & 0 & 0 & 0 & 0 \\
0 & 1 & 0 & 0 & 0 & 0 & 0 & 0 & 0 & 0 & 0 \\
0 & 0 & \frac{1}{2} & 0 & 0 & 0 & 0 & 0 & 0 & 0 & -\frac{3}{2} \\
0 & 0 & \frac{1}{2} & 0 & 0 & 0 & 0 & 0 & 0 & 0 & \frac{5}{2} \\
0 & 0 & 0 & 1 & 0 & 0 & 0 & 0 & 0 & 0 & 0 \\
0 & 0 & 0 & 0 & 1 & 0 & 0 & 0 & 0 & 0 & 0 \\
0 & 0 & 0 & 0 & 0 & 1 & 0 & 0 & 0 & 0 & 0 \\
0 & 0 & 0 & 0 & 0 & 0 & 1 & 0 & 0 & 0 & 0 \\
0 & 0 & 0 & 0 & 0 & 0 & 0 & 1 & 0 & 0 & 0 \\
0 & 0 & 0 & 0 & 0 & 0 & 0 & 0 & 1 & 0 & 0 \\
0 & 0 & 0 & 0 & 0 & 0 & 0 & 0 & 0 & 1 & 0
\end{array}\right)
\left(\begin{array}{rrrrrrrrrrr}
\frac{3}{4} & 0 & \frac{1}{4} & 0 & 0 & 0 & 0 & 0 & 0 & 0 & 0 \\
0 & 1 & 0 & 0 & 0 & 0 & 0 & 0 & 0 & 0 & 0 \\
0 & 0 & 0 & 1 & 0 & 0 & 0 & 0 & 0 & 0 & 0 \\
0 & 0 & 0 & 0 & 1 & 0 & 0 & 0 & 0 & 0 & 0 \\
0 & 0 & 0 & 0 & 0 & 1 & 0 & 0 & 0 & 0 & 0 \\
0 & 0 & 0 & 0 & 0 & 0 & 1 & 0 & 0 & 0 & 0 \\
\frac{1}{4} & 0 & \frac{3}{4} & 0 & 0 & 0 & 0 & 0 & 0 & 0 & 0 \\
0 & 0 & 0 & 0 & 0 & 0 & 0 & 1 & 0 & 0 & 0 \\
0 & 0 & 0 & 0 & 0 & 0 & 0 & 0 & 1 & 0 & 0 \\
0 & 0 & 0 & 0 & 0 & 0 & 0 & 0 & 0 & 1 & 0 \\
0 & 0 & 0 & 0 & 0 & 0 & 0 & 0 & 0 & 0 & 1
\end{array}\right) \\
\left(\begin{array}{rrrrrrrrrrr}
1 & 0 & 0 & 0 & 0 & 0 & 0 & 0 & 0 & 0 & 0 \\
0 & 1 & 0 & 0 & 0 & 0 & 0 & 0 & 0 & 0 & 0 \\
0 & 0 & 1 & 0 & 0 & 0 & 0 & 0 & 0 & 0 & 0 \\
0 & 0 & 0 & \frac{2}{3} & 0 & 0 & 0 & 0 & 0 & 0 & -1 \\
0 & 0 & 0 & \frac{1}{3} & 0 & 0 & 0 & 0 & 0 & 0 & 2 \\
0 & 0 & 0 & 0 & 1 & 0 & 0 & 0 & 0 & 0 & 0 \\
0 & 0 & 0 & 0 & 0 & 1 & 0 & 0 & 0 & 0 & 0 \\
0 & 0 & 0 & 0 & 0 & 0 & 1 & 0 & 0 & 0 & 0 \\
0 & 0 & 0 & 0 & 0 & 0 & 0 & 1 & 0 & 0 & 0 \\
0 & 0 & 0 & 0 & 0 & 0 & 0 & 0 & 1 & 0 & 0 \\
0 & 0 & 0 & 0 & 0 & 0 & 0 & 0 & 0 & 1 & 0
\end{array}\right)
\left(\begin{array}{rrrrrrrrrrr}
1 & 0 & 0 & 0 & 0 & 0 & 0 & 0 & 0 & 0 & 0 \\
0 & 1 & 0 & 0 & 0 & 0 & 0 & 0 & 0 & 0 & 0 \\
0 & 0 & 1 & 0 & 0 & 0 & 0 & 0 & 0 & 0 & 0 \\
0 & 0 & 0 & \frac{3}{4} & 0 & 0 & 0 & 0 & 0 & 0 & -\frac{3}{4} \\
0 & 0 & 0 & \frac{1}{4} & 0 & 0 & 0 & 0 & 0 & 0 & \frac{7}{4} \\
0 & 0 & 0 & 0 & 1 & 0 & 0 & 0 & 0 & 0 & 0 \\
0 & 0 & 0 & 0 & 0 & 1 & 0 & 0 & 0 & 0 & 0 \\
0 & 0 & 0 & 0 & 0 & 0 & 1 & 0 & 0 & 0 & 0 \\
0 & 0 & 0 & 0 & 0 & 0 & 0 & 1 & 0 & 0 & 0 \\
0 & 0 & 0 & 0 & 0 & 0 & 0 & 0 & 1 & 0 & 0 \\
0 & 0 & 0 & 0 & 0 & 0 & 0 & 0 & 0 & 1 & 0
\end{array}\right) \\
\left(\begin{array}{rrrrrrrrrrr}
1 & 0 & 0 & 0 & 0 & 0 & 0 & 0 & 0 & 0 & 0 \\
0 & 1 & 0 & 0 & 0 & 0 & 0 & 0 & 0 & 0 & 0 \\
0 & 0 & \frac{4}{5} & 0 & \frac{3}{5} & 0 & 0 & 0 & 0 & 0 & 0 \\
0 & 0 & 0 & 1 & 0 & 0 & 0 & 0 & 0 & 0 & 0 \\
0 & 0 & 0 & 0 & 0 & 1 & 0 & 0 & 0 & 0 & 0 \\
0 & 0 & 0 & 0 & 0 & 0 & 1 & 0 & 0 & 0 & 0 \\
0 & 0 & 0 & 0 & 0 & 0 & 0 & 1 & 0 & 0 & 0 \\
0 & 0 & 0 & 0 & 0 & 0 & 0 & 0 & 1 & 0 & 0 \\
0 & 0 & \frac{1}{5} & 0 & \frac{2}{5} & 0 & 0 & 0 & 0 & 0 & 0 \\
0 & 0 & 0 & 0 & 0 & 0 & 0 & 0 & 0 & 1 & 0 \\
0 & 0 & 0 & 0 & 0 & 0 & 0 & 0 & 0 & 0 & 1
\end{array}\right)
\left(\begin{array}{rrrrrrrrrrr}
1 & 0 & 0 & 0 & 0 & 0 & 0 & 0 & 0 & 0 & 0 \\
0 & 5 & 0 & 0 & 0 & 0 & 0 & 0 & -2 & 0 & 0 \\
0 & 0 & 1 & 0 & 0 & 0 & 0 & 0 & 0 & 0 & 0 \\
0 & -4 & 0 & 0 & 0 & 0 & 0 & 0 & 3 & 0 & 0 \\
0 & 0 & 0 & 1 & 0 & 0 & 0 & 0 & 0 & 0 & 0 \\
0 & 0 & 0 & 0 & 1 & 0 & 0 & 0 & 0 & 0 & 0 \\
0 & 0 & 0 & 0 & 0 & 1 & 0 & 0 & 0 & 0 & 0 \\
0 & 0 & 0 & 0 & 0 & 0 & 1 & 0 & 0 & 0 & 0 \\
0 & 0 & 0 & 0 & 0 & 0 & 0 & 1 & 0 & 0 & 0 \\
0 & 0 & 0 & 0 & 0 & 0 & 0 & 0 & 0 & 1 & 0 \\
0 & 0 & 0 & 0 & 0 & 0 & 0 & 0 & 0 & 0 & 1
\end{array}\right)
\end{gather*}

\begin{gather*}
\left(\begin{array}{rrrrrrrrrrr}
1 & 0 & 0 & 0 & 0 & 0 & 0 & 0 & 0 & 0 & 0 \\
0 & 1 & 0 & 0 & 0 & 0 & 0 & 0 & 0 & 0 & 0 \\
0 & 0 & \frac{5}{6} & 0 & \frac{1}{2} & 0 & 0 & 0 & 0 & 0 & 0 \\
0 & 0 & 0 & 1 & 0 & 0 & 0 & 0 & 0 & 0 & 0 \\
0 & 0 & 0 & 0 & 0 & 1 & 0 & 0 & 0 & 0 & 0 \\
0 & 0 & 0 & 0 & 0 & 0 & 1 & 0 & 0 & 0 & 0 \\
0 & 0 & 0 & 0 & 0 & 0 & 0 & 1 & 0 & 0 & 0 \\
0 & 0 & 0 & 0 & 0 & 0 & 0 & 0 & 1 & 0 & 0 \\
0 & 0 & 0 & 0 & 0 & 0 & 0 & 0 & 0 & 1 & 0 \\
0 & 0 & \frac{1}{6} & 0 & \frac{1}{2} & 0 & 0 & 0 & 0 & 0 & 0 \\
0 & 0 & 0 & 0 & 0 & 0 & 0 & 0 & 0 & 0 & 1
\end{array}\right)
\left(\begin{array}{rrrrrrrrrrr}
1 & 0 & 0 & 0 & 0 & 0 & 0 & 0 & 0 & 0 & 0 \\
0 & 1 & 0 & 0 & 0 & 0 & 0 & 0 & 0 & 0 & 0 \\
0 & 0 & \frac{6}{7} & 0 & 0 & \frac{3}{7} & 0 & 0 & 0 & 0 & 0 \\
0 & 0 & 0 & 1 & 0 & 0 & 0 & 0 & 0 & 0 & 0 \\
0 & 0 & 0 & 0 & 1 & 0 & 0 & 0 & 0 & 0 & 0 \\
0 & 0 & 0 & 0 & 0 & 0 & 1 & 0 & 0 & 0 & 0 \\
0 & 0 & 0 & 0 & 0 & 0 & 0 & 1 & 0 & 0 & 0 \\
0 & 0 & 0 & 0 & 0 & 0 & 0 & 0 & 1 & 0 & 0 \\
0 & 0 & 0 & 0 & 0 & 0 & 0 & 0 & 0 & 1 & 0 \\
0 & 0 & 0 & 0 & 0 & 0 & 0 & 0 & 0 & 0 & 1 \\
0 & 0 & \frac{1}{7} & 0 & 0 & \frac{4}{7} & 0 & 0 & 0 & 0 & 0
\end{array}\right)\\
\left(\begin{array}{rrrrrrrrrrr}
1 & 0 & 0 & 0 & 0 & 0 & 0 & 0 & 0 & 0 & 0 \\
0 & 1 & 0 & 0 & 0 & 0 & 0 & 0 & 0 & 0 & 0 \\
0 & 0 & 1 & 0 & 0 & 0 & 0 & 0 & 0 & 0 & 0 \\
0 & 0 & 0 & 1 & 0 & 0 & 0 & 0 & 0 & 0 & 0 \\
0 & 0 & 0 & 0 & 1 & 0 & 0 & 0 & 0 & 0 & 0 \\
0 & 0 & 0 & 0 & 0 & 1 & 0 & 0 & 0 & 0 & 0 \\
0 & 0 & 0 & 0 & 0 & 0 & 3 & \frac{1}{2} & 0 & 0 & 0 \\
0 & 0 & 0 & 0 & 0 & 0 & 0 & 0 & 1 & 0 & 0 \\
0 & 0 & 0 & 0 & 0 & 0 & 0 & 0 & 0 & 1 & 0 \\
0 & 0 & 0 & 0 & 0 & 0 & 0 & 0 & 0 & 0 & 1 \\
0 & 0 & 0 & 0 & 0 & 0 & -2 & \frac{1}{2} & 0 & 0 & 0
\end{array}\right)
\left(\begin{array}{rrrrrrrrrrr}
1 & 0 & 0 & 0 & 0 & 0 & 0 & 0 & 0 & 0 & 0 \\
0 & 1 & 0 & 0 & 0 & 0 & 0 & 0 & 0 & 0 & 0 \\
0 & 0 & 1 & 0 & 0 & 0 & 0 & 0 & 0 & 0 & 0 \\
0 & 0 & 0 & 1 & 0 & 0 & 0 & 0 & 0 & 0 & 0 \\
0 & 0 & 0 & 0 & 1 & 0 & 0 & 0 & 0 & 0 & 0 \\
0 & 0 & 0 & 0 & 0 & 1 & 0 & 0 & 0 & 0 & 0 \\
0 & 0 & 0 & 0 & 0 & 0 & 1 & 0 & 0 & 0 & 0 \\
0 & 0 & 0 & 0 & 0 & 0 & 0 & \frac{4}{3} & 0 & 0 & -\frac{2}{3} \\
0 & 0 & 0 & 0 & 0 & 0 & 0 & 0 & 1 & 0 & 0 \\
0 & 0 & 0 & 0 & 0 & 0 & 0 & 0 & 0 & 1 & 0 \\
0 & 0 & 0 & 0 & 0 & 0 & 0 & -\frac{1}{3} & 0 & 0 & \frac{5}{3}
\end{array}\right) \\
\left(\begin{array}{rrrrrrrrrrr}
1 & 0 & 0 & 0 & 0 & 0 & 0 & 0 & 0 & 0 & 0 \\
0 & \frac{1}{5} & 0 & 0 & 0 & 0 & 0 & 0 & 0 & -\frac{2}{5} & 0 \\
0 & \frac{4}{5} & 0 & 0 & 0 & 0 & 0 & 0 & 0 & \frac{7}{5} & 0 \\
0 & 0 & 1 & 0 & 0 & 0 & 0 & 0 & 0 & 0 & 0 \\
0 & 0 & 0 & 1 & 0 & 0 & 0 & 0 & 0 & 0 & 0 \\
0 & 0 & 0 & 0 & 1 & 0 & 0 & 0 & 0 & 0 & 0 \\
0 & 0 & 0 & 0 & 0 & 1 & 0 & 0 & 0 & 0 & 0 \\
0 & 0 & 0 & 0 & 0 & 0 & 1 & 0 & 0 & 0 & 0 \\
0 & 0 & 0 & 0 & 0 & 0 & 0 & 1 & 0 & 0 & 0 \\
0 & 0 & 0 & 0 & 0 & 0 & 0 & 0 & 1 & 0 & 0 \\
0 & 0 & 0 & 0 & 0 & 0 & 0 & 0 & 0 & 0 & 1
\end{array}\right)
\left(\begin{array}{rrrrrrrrrrr}
1 & 0 & 0 & 0 & 0 & 0 & 0 & 0 & 0 & 0 & 0 \\
0 & 1 & 0 & 0 & 0 & 0 & 0 & 0 & 0 & 0 & 0 \\
0 & 0 & 0 & 0 & 1 & 0 & 0 & 0 & 0 & 0 & 0 \\
0 & 0 & 0 & 0 & 0 & 1 & 0 & 0 & 0 & 0 & 0 \\
0 & 0 & \frac{7}{3} & 2 & 0 & 0 & 0 & 0 & 0 & 0 & 0 \\
0 & 0 & 0 & 0 & 0 & 0 & 1 & 0 & 0 & 0 & 0 \\
0 & 0 & 0 & 0 & 0 & 0 & 0 & 1 & 0 & 0 & 0 \\
0 & 0 & 0 & 0 & 0 & 0 & 0 & 0 & 1 & 0 & 0 \\
0 & 0 & 0 & 0 & 0 & 0 & 0 & 0 & 0 & 1 & 0 \\
0 & 0 & 0 & 0 & 0 & 0 & 0 & 0 & 0 & 0 & 1 \\
0 & 0 & -\frac{4}{3} & -1 & 0 & 0 & 0 & 0 & 0 & 0 & 0
\end{array}\right) \\
\left(\begin{array}{rrrrrrrrrrr}
1 & 0 & 0 & 0 & 0 & 0 & 0 & 0 & 0 & 0 & 0 \\
0 & 1 & 0 & 0 & 0 & 0 & 0 & 0 & 0 & 0 & 0 \\
0 & 0 & 1 & 0 & 0 & 0 & 0 & 0 & 0 & 0 & 0 \\
0 & 0 & 0 & 1 & 0 & 0 & 0 & 0 & 0 & 0 & 0 \\
0 & 0 & 0 & 0 & 1 & 0 & 0 & 0 & 0 & 0 & 0 \\
0 & 0 & 0 & 0 & 0 & 1 & 0 & 0 & 0 & 0 & 0 \\
0 & 0 & 0 & 0 & 0 & 0 & \frac{1}{3} & 0 & \frac{1}{6} & 0 & 0 \\
0 & 0 & 0 & 0 & 0 & 0 & 0 & 1 & 0 & 0 & 0 \\
0 & 0 & 0 & 0 & 0 & 0 & 0 & 0 & 0 & 1 & 0 \\
0 & 0 & 0 & 0 & 0 & 0 & 0 & 0 & 0 & 0 & 1 \\
0 & 0 & 0 & 0 & 0 & 0 & \frac{2}{3} & 0 & \frac{5}{6} & 0 & 0
\end{array}\right)
\left(\begin{array}{rrrrrrrrrrr}
1 & 0 & 0 & 0 & 0 & 0 & 0 & 0 & 0 & 0 & 0 \\
0 & 1 & 0 & 0 & 0 & 0 & 0 & 0 & 0 & 0 & 0 \\
0 & 0 & 1 & 0 & 0 & 0 & 0 & 0 & 0 & 0 & 0 \\
0 & 0 & 0 & 1 & 0 & 0 & 0 & 0 & 0 & 0 & 0 \\
0 & 0 & 0 & 0 & 1 & 0 & 0 & 0 & 0 & 0 & 0 \\
0 & 0 & 0 & 0 & 0 & 1 & 0 & 0 & 0 & 0 & 0 \\
0 & 0 & 0 & 0 & 0 & 0 & 1 & 0 & 0 & 0 & 0 \\
0 & 0 & 0 & 0 & 0 & 0 & 0 & 0 & 0 & -3 & -5 \\
0 & 0 & 0 & 0 & 0 & 0 & 0 & 0 & 0 & 4 & 6 \\
0 & 0 & 0 & 0 & 0 & 0 & 0 & 1 & 0 & 0 & 0 \\
0 & 0 & 0 & 0 & 0 & 0 & 0 & 0 & 1 & 0 & 0
\end{array}\right)
\end{gather*}

\begin{gather*}
\left(\begin{array}{rrrrrrrrrrr}
1 & 0 & 0 & 0 & 0 & 0 & 0 & 0 & 0 & 0 & 0 \\
0 & 1 & 0 & 0 & 0 & 0 & 0 & 0 & 0 & 0 & 0 \\
0 & 0 & \frac{3}{2} & 0 & 0 & 0 & 0 & 0 & 0 & -\frac{3}{2} & 0 \\
0 & 0 & 0 & 1 & 0 & 0 & 0 & 0 & 0 & 0 & 0 \\
0 & 0 & 0 & 0 & 1 & 0 & 0 & 0 & 0 & 0 & 0 \\
0 & 0 & -\frac{1}{2} & 0 & 0 & 0 & 0 & 0 & 0 & \frac{5}{2} & 0 \\
0 & 0 & 0 & 0 & 0 & 1 & 0 & 0 & 0 & 0 & 0 \\
0 & 0 & 0 & 0 & 0 & 0 & 1 & 0 & 0 & 0 & 0 \\
0 & 0 & 0 & 0 & 0 & 0 & 0 & 1 & 0 & 0 & 0 \\
0 & 0 & 0 & 0 & 0 & 0 & 0 & 0 & 1 & 0 & 0 \\
0 & 0 & 0 & 0 & 0 & 0 & 0 & 0 & 0 & 0 & 1
\end{array}\right)
\left(\begin{array}{rrrrrrrrrrr}
1 & 0 & 0 & 0 & 0 & 0 & 0 & 0 & 0 & 0 & 0 \\
0 & 1 & 0 & 0 & 0 & 0 & 0 & 0 & 0 & 0 & 0 \\
0 & 0 & 2 & 0 & 0 & 0 & 0 & 0 & -3 & 0 & 0 \\
0 & 0 & 0 & 1 & 0 & 0 & 0 & 0 & 0 & 0 & 0 \\
0 & 0 & -1 & 0 & 0 & 0 & 0 & 0 & 4 & 0 & 0 \\
0 & 0 & 0 & 0 & 1 & 0 & 0 & 0 & 0 & 0 & 0 \\
0 & 0 & 0 & 0 & 0 & 1 & 0 & 0 & 0 & 0 & 0 \\
0 & 0 & 0 & 0 & 0 & 0 & 1 & 0 & 0 & 0 & 0 \\
0 & 0 & 0 & 0 & 0 & 0 & 0 & 1 & 0 & 0 & 0 \\
0 & 0 & 0 & 0 & 0 & 0 & 0 & 0 & 0 & 1 & 0 \\
0 & 0 & 0 & 0 & 0 & 0 & 0 & 0 & 0 & 0 & 1
\end{array}\right)\\
\left(\begin{array}{rrrrrrrrrrr}
1 & 0 & 0 & 0 & 0 & 0 & 0 & 0 & 0 & 0 & 0 \\
0 & 1 & 0 & 0 & 0 & 0 & 0 & 0 & 0 & 0 & 0 \\
0 & 0 & 1 & 0 & 0 & 0 & 0 & 0 & 0 & 0 & 0 \\
0 & 0 & 0 & \frac{7}{6} & \frac{1}{2} & 0 & 0 & 0 & 0 & 0 & 0 \\
0 & 0 & 0 & 0 & 0 & 1 & 0 & 0 & 0 & 0 & 0 \\
0 & 0 & 0 & 0 & 0 & 0 & 1 & 0 & 0 & 0 & 0 \\
0 & 0 & 0 & 0 & 0 & 0 & 0 & 1 & 0 & 0 & 0 \\
0 & 0 & 0 & 0 & 0 & 0 & 0 & 0 & 1 & 0 & 0 \\
0 & 0 & 0 & 0 & 0 & 0 & 0 & 0 & 0 & 1 & 0 \\
0 & 0 & 0 & 0 & 0 & 0 & 0 & 0 & 0 & 0 & 1 \\
0 & 0 & 0 & -\frac{1}{6} & \frac{1}{2} & 0 & 0 & 0 & 0 & 0 & 0
\end{array}\right)
\left(\begin{array}{rrrrrrrrrrr}
1 & 0 & 0 & 0 & 0 & 0 & 0 & 0 & 0 & 0 & 0 \\
0 & 1 & 0 & 0 & 0 & 0 & 0 & 0 & 0 & 0 & 0 \\
0 & 0 & 1 & 0 & 0 & 0 & 0 & 0 & 0 & 0 & 0 \\
0 & 0 & 0 & \frac{6}{5} & \frac{3}{5} & 0 & 0 & 0 & 0 & 0 & 0 \\
0 & 0 & 0 & 0 & 0 & 1 & 0 & 0 & 0 & 0 & 0 \\
0 & 0 & 0 & 0 & 0 & 0 & 1 & 0 & 0 & 0 & 0 \\
0 & 0 & 0 & 0 & 0 & 0 & 0 & 1 & 0 & 0 & 0 \\
0 & 0 & 0 & 0 & 0 & 0 & 0 & 0 & 1 & 0 & 0 \\
0 & 0 & 0 & 0 & 0 & 0 & 0 & 0 & 0 & 1 & 0 \\
0 & 0 & 0 & 0 & 0 & 0 & 0 & 0 & 0 & 0 & 1 \\
0 & 0 & 0 & -\frac{1}{5} & \frac{2}{5} & 0 & 0 & 0 & 0 & 0 & 0
\end{array}\right) \\
\left(\begin{array}{rrrrrrrrrrr}
\frac{3}{2} & 0 & 0 & 0 & 0 & 0 & -\frac{1}{2} & 0 & 0 & 0 & 0 \\
0 & 1 & 0 & 0 & 0 & 0 & 0 & 0 & 0 & 0 & 0 \\
-\frac{1}{2} & 0 & 0 & 0 & 0 & 0 & \frac{3}{2} & 0 & 0 & 0 & 0 \\
0 & 0 & 1 & 0 & 0 & 0 & 0 & 0 & 0 & 0 & 0 \\
0 & 0 & 0 & 1 & 0 & 0 & 0 & 0 & 0 & 0 & 0 \\
0 & 0 & 0 & 0 & 1 & 0 & 0 & 0 & 0 & 0 & 0 \\
0 & 0 & 0 & 0 & 0 & 1 & 0 & 0 & 0 & 0 & 0 \\
0 & 0 & 0 & 0 & 0 & 0 & 0 & 1 & 0 & 0 & 0 \\
0 & 0 & 0 & 0 & 0 & 0 & 0 & 0 & 1 & 0 & 0 \\
0 & 0 & 0 & 0 & 0 & 0 & 0 & 0 & 0 & 1 & 0 \\
0 & 0 & 0 & 0 & 0 & 0 & 0 & 0 & 0 & 0 & 1
\end{array}\right)
\left(\begin{array}{rrrrrrrrrrr}
1 & 0 & 0 & 0 & 0 & 0 & 0 & 0 & 0 & 0 & 0 \\
0 & 1 & 0 & 0 & 0 & 0 & 0 & 0 & 0 & 0 & 0 \\
0 & 0 & \frac{5}{4} & \frac{3}{4} & 0 & 0 & 0 & 0 & 0 & 0 & 0 \\
0 & 0 & 0 & 0 & 1 & 0 & 0 & 0 & 0 & 0 & 0 \\
0 & 0 & 0 & 0 & 0 & 1 & 0 & 0 & 0 & 0 & 0 \\
0 & 0 & 0 & 0 & 0 & 0 & 1 & 0 & 0 & 0 & 0 \\
0 & 0 & 0 & 0 & 0 & 0 & 0 & 1 & 0 & 0 & 0 \\
0 & 0 & 0 & 0 & 0 & 0 & 0 & 0 & 1 & 0 & 0 \\
0 & 0 & 0 & 0 & 0 & 0 & 0 & 0 & 0 & 1 & 0 \\
0 & 0 & 0 & 0 & 0 & 0 & 0 & 0 & 0 & 0 & 1 \\
0 & 0 & -\frac{1}{4} & \frac{1}{4} & 0 & 0 & 0 & 0 & 0 & 0 & 0
\end{array}\right) \\
= \left(\begin{array}{rrrrrrrrrrr}
1 & \frac{4}{35} & 0 & 0 & 0 & \frac{4}{7} & \frac{2}{15} & 0 & 0 & 0 & \frac{2}{3} \\
\frac{4}{3} & 1 & -\frac{8}{3} & 0 & 0 & 0 & 0 & -4 & 0 & 0 & 0 \\
0 & -\frac{16}{35} & 1 & 0 & 0 & -\frac{16}{7} & -\frac{8}{15} & 0 & 0 & 0 & -\frac{8}{3} \\
0 & 0 & 0 & 1 & 0 & 0 & 0 & 0 & 0 & 0 & 0 \\
0 & 0 & 0 & 0 & 1 & 0 & 0 & 0 & 0 & 0 & 0 \\
-4 & 0 & 8 & 0 & 0 & 1 & 0 & 12 & 0 & 0 & 0 \\
-\frac{2}{3} & 0 & \frac{4}{3} & 0 & 0 & 0 & 1 & 2 & 0 & 0 & 0 \\
0 & \frac{12}{35} & 0 & 0 & 0 & \frac{12}{7} & \frac{2}{5} & 1 & 0 & 0 & 2 \\
0 & 0 & 0 & 0 & 0 & 0 & 0 & 0 & 1 & 0 & 0 \\
0 & 0 & 0 & 0 & 0 & 0 & 0 & 0 & 0 & 1 & 0 \\
\frac{10}{3} & 0 & -\frac{20}{3} & 0 & 0 & 0 & 0 & -10 & 0 & 0 & 1
\end{array}\right)
\end{gather*}

The braid $b_{kl}$ will correspond to the product of the following matrices:
$$
f_5(b_{kl})=A_{4625}A_{5716}A_{5817}A_{1645}A_{1756}A_{2578}A_{7815}A_{5627}A_{2758}A_{1567}A_{6725}A_{4516}A_{1657}A_{2546}.
$$

Let us denote the product of these matrices by the  $A_{b_{kl}}$. In matrix form this product looks as follow:
$$
A_{b_{kl}}=f_5(b_{kl})=
$$

\begin{gather*}
\left(\begin{array}{rrrrrrrrrrr}
1 & 0 & 0 & 0 & 0 & 0 & 0 & 0 & 0 & 0 & 0 \\
0 & 1 & 0 & 0 & 0 & 0 & 0 & 0 & 0 & 0 & 0 \\
0 & 0 & 1 & 0 & 0 & 0 & 0 & 0 & 0 & 0 & 0 \\
0 & 0 & 0 & 1 & 0 & 0 & 0 & 0 & 0 & 0 & 0 \\
0 & 0 & 0 & 0 & 1 & 0 & 0 & 0 & 0 & 0 & 0 \\
0 & 0 & 0 & 0 & 0 & 1 & 0 & 0 & 0 & 0 & 0 \\
0 & 0 & 0 & 0 & 0 & 0 & 1 & 0 & 0 & 0 & 0 \\
0 & 0 & 0 & 0 & 0 & 0 & 0 & \frac{1}{2} & 0 & 0 & -1 \\
0 & 0 & 0 & 0 & 0 & 0 & 0 & \frac{1}{2} & 0 & 0 & 2 \\
0 & 0 & 0 & 0 & 0 & 0 & 0 & 0 & 1 & 0 & 0 \\
0 & 0 & 0 & 0 & 0 & 0 & 0 & 0 & 0 & 1 & 0
\end{array}\right)
\left(\begin{array}{rrrrrrrrrrr}
1 & 0 & 0 & 0 & 0 & 0 & 0 & 0 & 0 & 0 & 0 \\
0 & 1 & 0 & 0 & 0 & 0 & 0 & 0 & 0 & 0 & 0 \\
0 & 0 & 1 & 0 & 0 & 0 & 0 & 0 & 0 & 0 & 0 \\
0 & 0 & 0 & \frac{1}{2} & 0 & 0 & 0 & 0 & 0 & 0 & -2 \\
0 & 0 & 0 & \frac{1}{2} & 0 & 0 & 0 & 0 & 0 & 0 & 3 \\
0 & 0 & 0 & 0 & 1 & 0 & 0 & 0 & 0 & 0 & 0 \\
0 & 0 & 0 & 0 & 0 & 1 & 0 & 0 & 0 & 0 & 0 \\
0 & 0 & 0 & 0 & 0 & 0 & 1 & 0 & 0 & 0 & 0 \\
0 & 0 & 0 & 0 & 0 & 0 & 0 & 1 & 0 & 0 & 0 \\
0 & 0 & 0 & 0 & 0 & 0 & 0 & 0 & 1 & 0 & 0 \\
0 & 0 & 0 & 0 & 0 & 0 & 0 & 0 & 0 & 1 & 0
\end{array}\right) \\
\left(\begin{array}{rrrrrrrrrrr}
1 & 0 & 0 & 0 & 0 & 0 & 0 & 0 & 0 & 0 & 0 \\
0 & 1 & 0 & 0 & 0 & 0 & 0 & 0 & 0 & 0 & 0 \\
0 & 0 & 1 & 0 & 0 & 0 & 0 & 0 & 0 & 0 & 0 \\
0 & 0 & 0 & \frac{2}{3} & 0 & 0 & 0 & 0 & 0 & 0 & -\frac{4}{3} \\
0 & 0 & 0 & \frac{1}{3} & 0 & 0 & 0 & 0 & 0 & 0 & \frac{7}{3} \\
0 & 0 & 0 & 0 & 1 & 0 & 0 & 0 & 0 & 0 & 0 \\
0 & 0 & 0 & 0 & 0 & 1 & 0 & 0 & 0 & 0 & 0 \\
0 & 0 & 0 & 0 & 0 & 0 & 1 & 0 & 0 & 0 & 0 \\
0 & 0 & 0 & 0 & 0 & 0 & 0 & 1 & 0 & 0 & 0 \\
0 & 0 & 0 & 0 & 0 & 0 & 0 & 0 & 1 & 0 & 0 \\
0 & 0 & 0 & 0 & 0 & 0 & 0 & 0 & 0 & 1 & 0
\end{array}\right)
\left(\begin{array}{rrrrrrrrrrr}
1 & 0 & 0 & 0 & 0 & 0 & 0 & 0 & 0 & 0 & 0 \\
0 & 1 & 0 & 0 & 0 & 0 & 0 & 0 & 0 & 0 & 0 \\
0 & 0 & \frac{4}{5} & \frac{3}{5} & 0 & 0 & 0 & 0 & 0 & 0 & 0 \\
0 & 0 & 0 & 0 & 1 & 0 & 0 & 0 & 0 & 0 & 0 \\
0 & 0 & 0 & 0 & 0 & 1 & 0 & 0 & 0 & 0 & 0 \\
0 & 0 & 0 & 0 & 0 & 0 & 1 & 0 & 0 & 0 & 0 \\
0 & 0 & 0 & 0 & 0 & 0 & 0 & 1 & 0 & 0 & 0 \\
0 & 0 & 0 & 0 & 0 & 0 & 0 & 0 & 1 & 0 & 0 \\
0 & 0 & \frac{1}{5} & \frac{2}{5} & 0 & 0 & 0 & 0 & 0 & 0 & 0 \\
0 & 0 & 0 & 0 & 0 & 0 & 0 & 0 & 0 & 1 & 0 \\
0 & 0 & 0 & 0 & 0 & 0 & 0 & 0 & 0 & 0 & 1
\end{array}\right) \\
\left(\begin{array}{rrrrrrrrrrr}
1 & 0 & 0 & 0 & 0 & 0 & 0 & 0 & 0 & 0 & 0 \\
0 & 1 & 0 & 0 & 0 & 0 & 0 & 0 & 0 & 0 & 0 \\
0 & 0 & 1 & 0 & 0 & 0 & 0 & 0 & 0 & 0 & 0 \\
0 & 0 & 0 & \frac{5}{6} & 0 & \frac{2}{3} & 0 & 0 & 0 & 0 & 0 \\
0 & 0 & 0 & 0 & 1 & 0 & 0 & 0 & 0 & 0 & 0 \\
0 & 0 & 0 & 0 & 0 & 0 & 1 & 0 & 0 & 0 & 0 \\
0 & 0 & 0 & 0 & 0 & 0 & 0 & 1 & 0 & 0 & 0 \\
0 & 0 & 0 & 0 & 0 & 0 & 0 & 0 & 1 & 0 & 0 \\
0 & 0 & 0 & 0 & 0 & 0 & 0 & 0 & 0 & 1 & 0 \\
0 & 0 & 0 & \frac{1}{6} & 0 & \frac{1}{3} & 0 & 0 & 0 & 0 & 0 \\
0 & 0 & 0 & 0 & 0 & 0 & 0 & 0 & 0 & 0 & 1
\end{array}\right)
\left(\begin{array}{rrrrrrrrrrr}
1 & 0 & 0 & 0 & 0 & 0 & 0 & 0 & 0 & 0 & 0 \\
0 & 1 & 0 & 0 & 0 & 0 & 0 & 0 & 0 & 0 & 0 \\
0 & 0 & 1 & 0 & 0 & 0 & 0 & 0 & 0 & 0 & 0 \\
0 & 0 & 0 & 1 & 0 & 0 & 0 & 0 & 0 & 0 & 0 \\
0 & 0 & 0 & 0 & 1 & 0 & 0 & 0 & 0 & 0 & 0 \\
0 & 0 & 0 & 0 & 0 & 1 & 0 & 0 & 0 & 0 & 0 \\
0 & 0 & 0 & 0 & 0 & 0 & 1 & 0 & 0 & 0 & 0 \\
0 & 0 & 0 & 0 & 0 & 0 & 0 & 1 & 0 & 0 & 0 \\
0 & 0 & 0 & 0 & 0 & 0 & 0 & 0 & 0 & 0 & 1 \\
0 & 0 & 0 & 0 & 0 & 0 & 0 & 0 & 2 & \frac{5}{3} & 0 \\
0 & 0 & 0 & 0 & 0 & 0 & 0 & 0 & -1 & -\frac{2}{3} & 0
\end{array}\right)
\end{gather*}
\begin{gather*}
\left(\begin{array}{rrrrrrrrrrr}
1 & 0 & 0 & 0 & 0 & 0 & 0 & 0 & 0 & 0 & 0 \\
0 & 1 & 0 & 0 & 0 & 0 & 0 & 0 & 0 & 0 & 0 \\
0 & 0 & 1 & 0 & 0 & 0 & 0 & 0 & 0 & 0 & 0 \\
0 & 0 & 0 & 0 & -2 & 0 & 0 & 0 & 0 & 0 & -6 \\
0 & 0 & 0 & 0 & 3 & 0 & 0 & 0 & 0 & 0 & 7 \\
0 & 0 & 0 & 1 & 0 & 0 & 0 & 0 & 0 & 0 & 0 \\
0 & 0 & 0 & 0 & 0 & 1 & 0 & 0 & 0 & 0 & 0 \\
0 & 0 & 0 & 0 & 0 & 0 & 1 & 0 & 0 & 0 & 0 \\
0 & 0 & 0 & 0 & 0 & 0 & 0 & 1 & 0 & 0 & 0 \\
0 & 0 & 0 & 0 & 0 & 0 & 0 & 0 & 1 & 0 & 0 \\
0 & 0 & 0 & 0 & 0 & 0 & 0 & 0 & 0 & 1 & 0
\end{array}\right)
\left(\begin{array}{rrrrrrrrrrr}
1 & 0 & 0 & 0 & 0 & 0 & 0 & 0 & 0 & 0 & 0 \\
0 & 1 & 0 & 0 & 0 & 0 & 0 & 0 & 0 & 0 & 0 \\
0 & 0 & 1 & 0 & 0 & 0 & 0 & 0 & 0 & 0 & 0 \\
0 & 0 & 0 & 1 & 0 & 0 & 0 & 0 & 0 & 0 & 0 \\
0 & 0 & 0 & 0 & 1 & 0 & 0 & 0 & 0 & 0 & 0 \\
0 & 0 & 0 & 0 & 0 & 1 & 0 & 0 & 0 & 0 & 0 \\
0 & 0 & 0 & 0 & 0 & 0 & 1 & 0 & 0 & 0 & 0 \\
0 & 0 & 0 & 0 & 0 & 0 & 0 & 2 & 0 & -3 & 0 \\
0 & 0 & 0 & 0 & 0 & 0 & 0 & 0 & 1 & 0 & 0 \\
0 & 0 & 0 & 0 & 0 & 0 & 0 & -1 & 0 & 4 & 0 \\
0 & 0 & 0 & 0 & 0 & 0 & 0 & 0 & 0 & 0 & 1
\end{array}\right) \\
\left(\begin{array}{rrrrrrrrrrr}
1 & 0 & 0 & 0 & 0 & 0 & 0 & 0 & 0 & 0 & 0 \\
0 & 1 & 0 & 0 & 0 & 0 & 0 & 0 & 0 & 0 & 0 \\
0 & 0 & 1 & 0 & 0 & 0 & 0 & 0 & 0 & 0 & 0 \\
0 & 0 & 0 & 1 & 0 & 0 & 0 & 0 & 0 & 0 & 0 \\
0 & 0 & 0 & 0 & 1 & 0 & 0 & 0 & 0 & 0 & 0 \\
0 & 0 & 0 & 0 & 0 & 1 & 0 & 0 & 0 & 0 & 0 \\
0 & 0 & 0 & 0 & 0 & 0 & 1 & 0 & 0 & 0 & 0 \\
0 & 0 & 0 & 0 & 0 & 0 & 0 & 1 & 0 & 0 & 0 \\
0 & 0 & 0 & 0 & 0 & 0 & 0 & 0 & \frac{6}{5} & \frac{3}{5} & 0 \\
0 & 0 & 0 & 0 & 0 & 0 & 0 & 0 & 0 & 0 & 1 \\
0 & 0 & 0 & 0 & 0 & 0 & 0 & 0 & -\frac{1}{5} & \frac{2}{5} & 0
\end{array}\right)
\left(\begin{array}{rrrrrrrrrrr}
1 & 0 & 0 & 0 & 0 & 0 & 0 & 0 & 0 & 0 & 0 \\
0 & 1 & 0 & 0 & 0 & 0 & 0 & 0 & 0 & 0 & 0 \\
0 & 0 & 1 & 0 & 0 & 0 & 0 & 0 & 0 & 0 & 0 \\
0 & 0 & 0 & \frac{3}{2} & \frac{5}{4} & 0 & 0 & 0 & 0 & 0 & 0 \\
0 & 0 & 0 & 0 & 0 & 1 & 0 & 0 & 0 & 0 & 0 \\
0 & 0 & 0 & 0 & 0 & 0 & 1 & 0 & 0 & 0 & 0 \\
0 & 0 & 0 & 0 & 0 & 0 & 0 & 1 & 0 & 0 & 0 \\
0 & 0 & 0 & 0 & 0 & 0 & 0 & 0 & 1 & 0 & 0 \\
0 & 0 & 0 & 0 & 0 & 0 & 0 & 0 & 0 & 1 & 0 \\
0 & 0 & 0 & 0 & 0 & 0 & 0 & 0 & 0 & 0 & 1 \\
0 & 0 & 0 & -\frac{1}{2} & -\frac{1}{4} & 0 & 0 & 0 & 0 & 0 & 0
\end{array}\right) \\
\left(\begin{array}{rrrrrrrrrrr}
1 & 0 & 0 & 0 & 0 & 0 & 0 & 0 & 0 & 0 & 0 \\
0 & 1 & 0 & 0 & 0 & 0 & 0 & 0 & 0 & 0 & 0 \\
0 & 0 & 1 & 0 & 0 & 0 & 0 & 0 & 0 & 0 & 0 \\
0 & 0 & 0 & 1 & 0 & 0 & 0 & 0 & 0 & 0 & 0 \\
0 & 0 & 0 & 0 & 1 & 0 & 0 & 0 & 0 & 0 & 0 \\
0 & 0 & 0 & 0 & 0 & 1 & 0 & 0 & 0 & 0 & 0 \\
0 & 0 & 0 & 0 & 0 & 0 & 1 & 0 & 0 & 0 & 0 \\
0 & 0 & 0 & 0 & 0 & 0 & 0 & 1 & 0 & 0 & 0 \\
0 & 0 & 0 & 0 & 0 & 0 & 0 & 0 & -1 & 0 & -4 \\
0 & 0 & 0 & 0 & 0 & 0 & 0 & 0 & 2 & 0 & 5 \\
0 & 0 & 0 & 0 & 0 & 0 & 0 & 0 & 0 & 1 & 0
\end{array}\right)
\left(\begin{array}{rrrrrrrrrrr}
1 & 0 & 0 & 0 & 0 & 0 & 0 & 0 & 0 & 0 & 0 \\
0 & 1 & 0 & 0 & 0 & 0 & 0 & 0 & 0 & 0 & 0 \\
0 & 0 & 2 & 0 & 0 & 0 & 0 & 0 & 0 & -3 & 0 \\
0 & 0 & -1 & 0 & 0 & 0 & 0 & 0 & 0 & 4 & 0 \\
0 & 0 & 0 & 1 & 0 & 0 & 0 & 0 & 0 & 0 & 0 \\
0 & 0 & 0 & 0 & 1 & 0 & 0 & 0 & 0 & 0 & 0 \\
0 & 0 & 0 & 0 & 0 & 1 & 0 & 0 & 0 & 0 & 0 \\
0 & 0 & 0 & 0 & 0 & 0 & 1 & 0 & 0 & 0 & 0 \\
0 & 0 & 0 & 0 & 0 & 0 & 0 & 1 & 0 & 0 & 0 \\
0 & 0 & 0 & 0 & 0 & 0 & 0 & 0 & 1 & 0 & 0 \\
0 & 0 & 0 & 0 & 0 & 0 & 0 & 0 & 0 & 0 & 1
\end{array}\right) \\
\left(\begin{array}{rrrrrrrrrrr}
1 & 0 & 0 & 0 & 0 & 0 & 0 & 0 & 0 & 0 & 0 \\
0 & 1 & 0 & 0 & 0 & 0 & 0 & 0 & 0 & 0 & 0 \\
0 & 0 & 1 & 0 & 0 & 0 & 0 & 0 & 0 & 0 & 0 \\
0 & 0 & 0 & \frac{6}{5} & \frac{4}{5} & 0 & 0 & 0 & 0 & 0 & 0 \\
0 & 0 & 0 & 0 & 0 & 1 & 0 & 0 & 0 & 0 & 0 \\
0 & 0 & 0 & 0 & 0 & 0 & 1 & 0 & 0 & 0 & 0 \\
0 & 0 & 0 & 0 & 0 & 0 & 0 & 1 & 0 & 0 & 0 \\
0 & 0 & 0 & 0 & 0 & 0 & 0 & 0 & 1 & 0 & 0 \\
0 & 0 & 0 & 0 & 0 & 0 & 0 & 0 & 0 & 1 & 0 \\
0 & 0 & 0 & 0 & 0 & 0 & 0 & 0 & 0 & 0 & 1 \\
0 & 0 & 0 & -\frac{1}{5} & \frac{1}{5} & 0 & 0 & 0 & 0 & 0 & 0
\end{array}\right)
\left(\begin{array}{rrrrrrrrrrr}
1 & 0 & 0 & 0 & 0 & 0 & 0 & 0 & 0 & 0 & 0 \\
0 & 1 & 0 & 0 & 0 & 0 & 0 & 0 & 0 & 0 & 0 \\
0 & 0 & 1 & 0 & 0 & 0 & 0 & 0 & 0 & 0 & 0 \\
0 & 0 & 0 & 1 & 0 & 0 & 0 & 0 & 0 & 0 & 0 \\
0 & 0 & 0 & 0 & 1 & 0 & 0 & 0 & 0 & 0 & 0 \\
0 & 0 & 0 & 0 & 0 & 1 & 0 & 0 & 0 & 0 & 0 \\
0 & 0 & 0 & 0 & 0 & 0 & 1 & 0 & 0 & 0 & 0 \\
0 & 0 & 0 & 0 & 0 & 0 & 0 & \frac{4}{3} & \frac{2}{3} & 0 & 0 \\
0 & 0 & 0 & 0 & 0 & 0 & 0 & 0 & 0 & 1 & 0 \\
0 & 0 & 0 & 0 & 0 & 0 & 0 & 0 & 0 & 0 & 1 \\
0 & 0 & 0 & 0 & 0 & 0 & 0 & -\frac{1}{3} & \frac{1}{3} & 0 & 0
\end{array}\right)
\end{gather*}

$$
=\left(\begin{array}{rrrrrrrrrrr}
1 & 0 & 0 & 0 & 0 & 0 & 0 & 0 & 0 & 0 & 0 \\
0 & 1 & 0 & 0 & 0 & 0 & 0 & 0 & 0 & 0 & 0 \\
0 & 0 & 1 & 0 & 1 & -1 & 0 & 0 & 0 & \frac{6}{5} & -\frac{6}{5} \\
0 & 0 & 0 & 1 & -\frac{5}{3} & \frac{5}{3} & 0 & 0 & 0 & -2 & 2 \\
0 & 0 & -\frac{5}{2} & \frac{5}{2} & 1 & 0 & 0 & -\frac{10}{3} & \frac{10}{3} & 0 & 0 \\
0 & 0 & \frac{7}{2} & -\frac{7}{2} & 0 & 1 & 0 & \frac{14}{3} & -\frac{14}{3} & 0 & 0 \\
0 & 0 & 0 & 0 & 0 & 0 & 1 & 0 & 0 & 0 & 0 \\
0 & 0 & 0 & 0 & -\frac{2}{3} & \frac{2}{3} & 0 & 1 & 0 & -\frac{4}{5} & \frac{4}{5} \\
0 & 0 & 0 & 0 & \frac{4}{3} & -\frac{4}{3} & 0 & 0 & 1 & \frac{8}{5} & -\frac{8}{5} \\
0 & 0 & 2 & -2 & 0 & 0 & 0 & \frac{8}{3} & -\frac{8}{3} & 1 & 0 \\
0 & 0 & -3 & 3 & 0 & 0 & 0 & -4 & 4 & 0 & 1
\end{array}\right)
$$

\bigbreak
After which we can easily verify that the desired relation is satisfied:
\bigbreak
\resizebox{0.9\hsize}{!}{$
A_{b_{ij}}A_{b_{kl}}=A_{b_{kl}}A_{b_{ij}}=\left(\begin{array}{rrrrrrrrrrr}
1 & \frac{4}{35} & 0 & 0 & 0 & \frac{4}{7} & \frac{2}{15} & 0 & 0 & 0 & \frac{2}{3} \\
\frac{4}{3} & 1 & -\frac{8}{3} & 0 & 0 & 0 & 0 & -4 & 0 & 0 & 0 \\
0 & -\frac{16}{35} & 1 & 0 & 1 & -\frac{23}{7} & -\frac{8}{15} & 0 & 0 & \frac{6}{5} & -\frac{58}{15} \\
0 & 0 & 0 & 1 & -\frac{5}{3} & \frac{5}{3} & 0 & 0 & 0 & -2 & 2 \\
0 & 0 & -\frac{5}{2} & \frac{5}{2} & 1 & 0 & 0 & -\frac{10}{3} & \frac{10}{3} & 0 & 0 \\
-4 & 0 & \frac{23}{2} & -\frac{7}{2} & 0 & 1 & 0 & \frac{50}{3} & -\frac{14}{3} & 0 & 0 \\
-\frac{2}{3} & 0 & \frac{4}{3} & 0 & 0 & 0 & 1 & 2 & 0 & 0 & 0 \\
0 & \frac{12}{35} & 0 & 0 & -\frac{2}{3} & \frac{50}{21} & \frac{2}{5} & 1 & 0 & -\frac{4}{5} & \frac{14}{5} \\
0 & 0 & 0 & 0 & \frac{4}{3} & -\frac{4}{3} & 0 & 0 & 1 & \frac{8}{5} & -\frac{8}{5} \\
0 & 0 & 2 & -2 & 0 & 0 & 0 & \frac{8}{3} & -\frac{8}{3} & 1 & 0 \\
\frac{10}{3} & 0 & -\frac{29}{3} & 3 & 0 & 0 & 0 & -14 & 4 & 0 & 1
\end{array}\right).
$}

\bigbreak
All other relations given in Definition \ref{def:pure_braid_representation} are verified similarly.

\end{appendices}

\setcounter{secnumdepth}{0}

\pagebreak

\end{document}